\documentclass[11pt,a4paper,openany,oneside]{article}
\usepackage[a4paper,left=2.5cm,right=2.5cm, top=3cm, bottom=3cm]{geometry}
\usepackage{times}
\usepackage{authblk}
\usepackage[latin1]{inputenc}
\usepackage[russian, english]{babel}
\usepackage{upgreek}
\usepackage{mathrsfs}
\usepackage{stmaryrd}
\usepackage{amssymb,amsmath,amsthm}
\usepackage{hyperref}
\usepackage{eurosym}
\usepackage{color}
\usepackage{enumerate}
\usepackage{multirow}
\usepackage{mathtools}
\usepackage{scalerel}
\usepackage{booktabs}
\usepackage{float}
\usepackage[T2A,T1]{fontenc}
\usepackage{amsfonts}
\usepackage{cases}
\usepackage{thmtools}

\newtheorem{thm}{Theorem}[section]
\newcounter{maintheorem}

\newtheorem{mainth}[maintheorem]{Theorem}

\newtheorem{lem}[thm]{Lemma}
\newtheorem{lemma}[thm]{Lemma}
\newtheorem{prop}[thm]{Proposition}

\theoremstyle{definition}
\newtheorem{defn}{Definition}[section]
\theoremstyle{remark}
\newtheorem{remark}{Remark}
\definecolor{myred}{rgb}{0.84,0.07,0.14}

\newcommand{\R}{\mathbb R}
\newcommand{\N}{\mathbb N}

\newcommand{\les}{\lesssim}
\newcommand{\intN}{\int_{\R^N}}

\newcommand{\divv}{{\rm div}}

\def\e{{\rm e}}

\newcommand{\cM}{{\mathcal M}}

\newcommand{\cR}{{\mathcal R}}

 \def\dd{\, {\rm d}}

\newcommand{\CyrB}{\mbox{\usefont{T2A}{\rmdefault}{m}{n}\CYRB}}

\newcommand{\tb}{{\widetilde b}}
\newcommand{\tp}{{\widetilde p}}
\newcommand{\tQ}{{\widetilde Q}}
\newcommand{\tw}{{\widetilde w}}

\newcommand{\ff}{\mathfrak f}

\DeclareMathOperator{\rad}{rad}

\DeclareOldFontCommand{\it}{\normalfont\itshape}{\mathit}

\newcommand{\Erad}{E_{\text{rad}}}
\newcommand{\ges}{\gtrsim}

\numberwithin{equation}{section}

\begin{document}

\renewcommand{\thefootnote}{\fnsymbol{footnote}}
\renewcommand{\thefootnote}{\fnsymbol{footnote}}
\renewcommand{\thefootnote}{\arabic{footnote}}

\title{Schr\"odinger-Poisson systems with zero mass\\ in the Sobolev limiting case}

\date{}
\author{Giulio Romani}%

\affil{\small Dipartimento di Scienza e Alta Tecnologia \protect\\ Universit\`{a} degli Studi dell'Insubria\protect\\ and\protect\\ RISM-Riemann International School of Mathematics\protect\\ Villa Toeplitz, Via G.B. Vico, 46 - 21100 Varese, Italy \protect\\\texttt{giulio.romani@uninsubria.it}}

\maketitle

%-----------------------------------------------
%-----------------------------------------------

\begin{abstract} We study the existence of positive solutions for a class of systems which strongly couple a quasilinear Schr\"odinger equation driven by a weighted $N$-Laplace operator and without the mass term, and a higher-order fractional Poisson equation. Since the system is considered in $\R^N$, the limiting case for the Sobolev embedding, we consider nonlinearities with exponential growth. Existence is proved relying on the study of a corresponding Choquard equation in which the Riesz kernel is a sign-changing logarithm. This is in turn solved by means of a variational approximating procedure for an auxiliary Choquard equation where the logarithm is uniformly approximated by polynomial kernels.
\end{abstract}

\section{Introduction}

In this paper we investigate existence of solutions for the quasilinear Schr\"odinger-Poisson system in the whole space given by 
\begin{equation*}\tag{SP$_0$}\label{SP_0}
	\begin{cases}
		-\divv\left(A(|x|)|\nabla u|^{N-2}\nabla u\right)=Q(|x|)\phi\,f(u)\ \  &\mbox{in}\ \,\R^N,\\
		(-\Delta)^{\frac N2}\phi=Q(|x|)F(u) & \mbox{in}\ \,\R^N,\\
	\end{cases}
\end{equation*}
where $N\geq2$, the positive weight functions $A$ and $Q$ enjoy a suitable behaviour at $0$ and infinity, $f:\R\to\R$ is a positive nonlinearity, and $F(s):=\int_0^sf(t)\dd t$ its primitive function. The first evident peculiarity of this system is that the mass term is missing in the left-hand side of the Schr\"odinger equation, namely we are in the so-called "zero-mass case". Moreover, the quasilinear operator in the same equation is a weighted $p$-Laplacian with $p=N$, and since we are studying \eqref{SP_0} in the whole space $\R^N$, we are lead to a functional setting which is critical with respect to the Sobolev embeddings. This implies that one can consider nonlinearities $f$ with exponential growth, and that the kernel of the (possibly high-order, possibly fractional) operator in the Poisson equation is logarithmic. It is then of interest the study of the interplay of all these phenomena.
\vskip0.2truecm

Systems posed in $\R^N$ which couple a Schr\"odinger-type equation with a Poisson equation, emerge in several fields of Physics: in electrostatics, they model the interaction of two identically charged particles; in quantum mechanics, the self-interaction of the wave function with its own gravitational field; they also emerge in the Hartree theory for crystals and in astrophysics about selfgravitating boson stars; see \cite{BF,LRZ} and references therein. In the higher-dimensional case $N\geq3$, and when $f$ has a polynomial growth, there is a extensive literature about Schr\"odinger-Poisson system, see the survey \cite{MV} and references therein. The usual first step to deal with such systems is to transform them into equivalent Choquard equations, which are nonlocal in their nonlinear part, by solving the Poisson equation by means of the Riesz kernel, and inserting it into the Schr\"odinger equation. Besides the effect of variables reduction, the advantage of this approach is that Choquard equations can be studied by variational techniques. This reduction principle is well-known and widely used in case the Riesz kernel is polynomial. For instance, if we consider the system with positive potential $V$ given by
\begin{equation*}\tag{SP}\label{SP}
	\begin{cases}
		-\Delta u+V(x)u=\phi\,f(u)\ \  &\mbox{in}\ \,\R^N,\\
		-\Delta\phi=F(u) & \mbox{in}\ \,\R^N,\\
	\end{cases}
\end{equation*}
when $N\geq3$, the (positive) Riesz kernel of the Poisson equation is $c_N|x|^{2-N}$, where $c_N$ is an explicit positive constant depending on the dimension $N$, and the corresponding Choquard equation is then
\begin{equation}\label{Choq_poly}
	-\Delta u+V(x)u=c_N\left(\frac1{|x|^\mu}\ast F(u)\right)f(u)\quad\ \mbox{in}\ \ \R^N
\end{equation}
with $\mu=N-2$. Here, new interesting phenomena arise, such as the appearance of a lower-critical exponent in addition to the usual upper-critical exponent, as in the Sobolev case, see \cite{MV1,MV,CZ,CVZ}.% and references therein.

In the Sobolev limiting case the Riesz kernel is logarithmic and therefore unbounded both from below and from above. This, together with the fact that it is sign-changing, introduces a major difficulty with respect to the higher dimensional setting. Indeed, the Choquard equation which is formally related to \eqref{SP} when $N=2$ is
\begin{equation}\label{Ch}\tag{Ch}
	-\Delta u+V(x)u=\frac1{2\pi}\left(\log\frac1{|x|}\ast F(u)\right)f(u)\quad\ \mbox{in}\ \ \R^2.
\end{equation}
However, due to the convolution term, the functional associated to \eqref{Ch} is not well-defined in the natural space $H^1(\R^2)$.%, even when the mass term $|u|^{N-2}u$ is present in the left-hand side, and $A\equiv1\equiv Q$, and even in the planar case $N=2$, where the leading operator is linear, that is
%\begin{equation}\label{Ch}
%	-\Delta u+u=C_N\left(\log\frac1{|\cdot|}\ast F(u)\right)f(u)\quad\ \mbox{in}\ \R^2\,.
%\end{equation}
%\vskip0.2truecm

In this setting, if $f(u)=u$, an approach originating from the unpublished work of Stubbe \cite{Stubbe} was proposed in \cite{CW,DW,BCV,CW2}, according to which \eqref{Ch} is solved in a constraint subspace of $H^1(\R^2)$, where the logarithmic convolution term is well-defined. 

However, in dimension two it is well-known that the maximal degree of summability for functions belonging to $H^1(\R^2)$ is exponential. The more general case of $f$ with exponential growth, was considered in \cite{CT}, where the authors establish a proper functional setting by means of a log-weighted version of the Poho\v zaev-Trudinger inequality, so that the functional associated to \eqref{Ch} turns out to be well-defined. An extension of these techniques to a quasilinear extension of \eqref{Ch} has been recently established in \cite{BCT}, where also the relationship between the Choquard equation and the corresponding Schr\"odinger-Poisson system has been carefully analysed; see also \cite{CMR} for the study of a weighted version of \eqref{Ch} with similar techniques. We also refer to \cite{ACTY} in which exponential nonlinearities in \eqref{Choq_poly}, still mantaining the polynomial Riesz kernel and so loosing the connection to the Schr\"odinger-Poisson system, have been first investigated, see also \cite{dACS,AFS} in the case of weights; on the other hand, Schr\"odinger-Poisson systems with logaritmic kernel and exponential nonlinearity, but not in gradient form as \eqref{SP}, have been studied in \cite{AF,BM1}, see also \cite{BM2}. 

The method recently proposed by \cite{LRTZ} to study \eqref{Ch}, refined in \cite{CDL} and extended to the nonlinear fractional setting in \cite{CLR}, exploits instead an approximation approach: the logarithmic Riesz kernel in \eqref{Ch} is replaced by an approximating one built upon the classical polynomial Riesz kernel, and the strategy is to find a solution of \eqref{Ch} as the limit of a family of solutions of the approximating problems depending on a small parameter. The main advantage of this method is that one can work in the natural Sobolev space associated to the equation and exploit simpler variational tools, while the price to pay is of course a more careful analysis, since at the end one needs to be able to pass to the limit.

\vskip0.2truecm

The special case of an identically zero potential in Schr\"odinger equations and in systems such as \eqref{SP}, is in the literature referred to as "zero-mass case", and emerges in physical context, e.g. in the nonabelian gauge theory of particle physics, such as the study of the Yang-Mills equation, see \cite{Gi}. The natural framework to study zero-mass problems is the homogeneous space $D^{1,2}_0(\R^N)$ defined  as the completion of
$C^\infty_0(\R^N)$ with respect to the norm $\|u\|_{D^{1,2}_0(\R^N)}:=\left(\intN|\nabla u|^2\right)^\frac12$. This is an appropriate setting if $N\geq3$, thanks to the continuous embedding $D^{1,2}_0(\R^N)\hookrightarrow L^{2^*}(\R^N)$ with $2^*:=\tfrac{2N}{N-2}$. In this framework Schr\"odinger equations with zero mass have been investigated since the seminal work \cite{BL}, see e.g. \cite{AP, ASM}, while Choquard equations with zero mass, namely $V=0$ in \eqref{Choq_poly}, have been considered in \cite{AY}. On the other hand, the zero-mass \textit{planar} case is notoriously difficult to approach, due to the reason that the natural space $D^{1,2}_0(\R^2)$ is not well-defined as a space of functions because of the lack of any kind of embedding into $L^p(\R^2)$, $p\geq1$.

A possible way to overcome the drawback of working with the spaces $D^{1,N}(\R^N)$ has been noticed in the recent paper \cite{dAC} by means of weight functions. Here, the authors study the (possibly quasilinear) Schr\"odinger equation with zero mass in the limiting Sobolev case
\begin{equation*}\tag{S$_0$}\label{S_0}
	-\divv\left(A(|x|)|\nabla u|^{N-2}\nabla u\right)=Q(|x|)f(u)\quad\mbox{in}\ \,\R^N,
\end{equation*}
where $N\geq2$. Basing on Hardy-type inequalities, the Authors found conditions on the weight functions $A$ and $Q$ (see assumptions (A) and (Q) below), and in turns a suitable functional setting, so that \eqref{S_0} admits a nontrivial radially symmetric solution. In particular they were able to recover some compact embeddings into appropriate weighted Lebesgue spaces, in order to retrieve the possibility of using standard variational techniques.

Schr\"odinger-Poisson systems with zero mass in the limiting Sobolev case, namely \eqref{SP} with $V=0$, have been investigated to our knowledge only in the peculiar case $f(u)=u$, see e.g. \cite{WCR,ChT,CSTW,LRZ}. In these works the approach of \cite{CW} to work in the constrained Hilbert space $\{u\in H^1(\R^2)\,|\, \int_{\R^2}\log(2+|x|)u^2(x)\dd x<+\infty\}$ was followed and therefore one takes advantage of the positive part of the logarithm, which takes the place of a sort of a mass term. However, their method cannot be applied in the more general case of a nonlinearity $f$.
\vskip0.2truecm
Taking advantage of the functional setting developed in \cite{dAC}, we aim to investigate (possibly quasilinear) Schr\"odinger-Poisson systems with zero mass in the limiting Sobolev case \eqref{SP_0} with more general nonlinearities, with a possibly critical growth. 

The functional space which naturally arises from \eqref{S_0} is
$$E:=\Big\{u\in L^N_{loc}(\R^N)\,\Big|\,\int_{\R^N}A(|x|)|\nabla u|^N<+\infty\Big\},$$
which, endowed with the norm
$$\|u\|:=\left(\intN A(|x|)|\nabla u|^N\right)^\frac1N,$$
is a Banach space provided the function $A$ satisfies the condition 
\begin{enumerate}
	\item[(A)] $A:\R^+\to\R$ is continuous, $\liminf_{r\to0^+}A(r)>0$ and there exist $A_0,\ell>0$ such that $A(r)\geq A_0 r^\ell$, for all $r>0$,
\end{enumerate}
see \cite[Lemma 2.1 and Corollary 1.5]{dAC}. The subset of $E$ composed by radially symmetric functions is denoted by $\Erad$. 

For $p\geq 1$, we also define the $Q$-weighted Lebesgue space
$$L^p_Q(\R^N):=\Big\{u\in\cM(\R^N)\,\Big|\,\intN Q(|x|)|u|^p<+\infty\Big\}\,,$$
where $\cM(\R^N)$ stands for the set of all measurable functions on $\R^N$.

\begin{mainth}(\cite{dAC}, Theorem 1.2)\label{Thm_cpt_emb}
	Assume (A) and 
	\begin{enumerate}
		\item[(Q)] $Q:\R^+\to\R^+$ is continuous and there exist $b_0,b>-N$ such that
		$$\limsup_{r\to0^+}\frac{Q(r)}{r^{b_0}}<+\infty\quad\mbox{and}\quad\limsup_{r\to+\infty}\frac{Q(r)}{r^b}<+\infty\,.$$
	\end{enumerate}
	Then the embedding $\Erad\hookrightarrow L^p_Q(\R^N)$ is continuous for $\gamma\leq p<+\infty$, where
	\begin{equation}\label{gamma}
		\gamma:=\max\left\{N,\frac{(b-\ell+N)(N+1)}\ell+N\right\}=\begin{cases}
			N\ \  &\mbox{if}\ \,b<\ell-N,\\
			\frac{(b-\ell+N)(N+1)}\ell+N & \mbox{if}\ \,b\geq\ell-N.\\
		\end{cases}
	\end{equation}
	Furthermore, the embedding is compact for $\gamma\leq p<+\infty$ when $b<\ell-N$, and for $\gamma<p<+\infty$ when $b\geq\ell-N$.
\end{mainth}

Note that assumption (Q) allows weight functions which can be singular at the origin and vanishing at infinity, and has been used also in the study of Choquard equation with vanishing potential, see e.g. \cite{AFS}.

Thanks to the above results, de Albuquerque and Carvalho were able to prove in \cite{dAC} a suitable Trudinger-Moser inequality for functions in $E$ (see Theorem \ref{ThmAC_TM} below) and in turn the existence of a nontrivial radial solution for \eqref{S_0} with nonlinearities $f$ with exponential growth, and which behave like $o(s^{\gamma-1})$ near $0$. We note that a global growth condition is also required, namely 
\begin{equation}\label{global_ass}
	F(s)\geq\lambda s^\nu\qquad\mbox{with}\ \ \nu>\gamma\ \ \mbox{and}\ \  \lambda\ \ \mbox{large enough.}
\end{equation}
%\vskip0.2truecm
Inspired by the analysis in \cite{dAC}, we aim to exploit their functional setting in order to investigate existence for the Schr\"odinger-Poisson system \eqref{SP_0}. In fact, since the Riesz kernel of the Poisson equation in \eqref{SP_0} is
\begin{equation}\label{Riesz_log}
	I_N(x):=C_N\log\frac1{|x|}\qquad\mbox{with}\quad C_N^{-1}:= 2^{N-1}\pi^{\frac N2}\Gamma\left(\tfrac N2\right),
\end{equation}
we first study variationally the Choquard equation with zero mass and exponential nonlinearities
\begin{equation}\label{Choq_log}\tag{Ch$_0$}
	-\divv\left(A(|x|)|\nabla u|^{N-2}\nabla u\right)=C_N\left(\log\frac1{|x|}\ast Q(|x|)F(u)\right)Q(|x|)\,f(u)\quad\ \mbox{in}\ \ \R^N
\end{equation}
via an approximation strategy in the spirit of \cite{LRTZ}, and then prove rigorously that from a solution of \eqref{Choq_log} one may obtain a corresponding solution of the Schr\"odinger-Poisson system \eqref{SP_0}. To this end, a characterisation of the distributional solutions of the possibly higher-order, possibly fractional Poisson equation $(-\Delta)^{\frac N2}\phi=\ff$ in $\R^N$ obtained in \cite{H} will be essential.
\vskip0.2truecm
\paragraph{Overview} In the remaining part of this section, we present the precise formulation of our assumptions and our main results; moreover we give a more detailed glimpse about the variational approximating procedure we will follow. A short Section \ref{Prel} in which we discuss consequences of our assumptions together with some preliminaries comes next. In Section \ref{Sec_Riesz_log}, which is the core of the paper, we prove existence for the Choquard equation \eqref{Choq_log}. Finally in Section \ref{Sec_SP} we derive from it the existence result for the Schr\"odinger-Poisson system \eqref{SP_0}.

\paragraph{\textbf{Notation.}} For $R>0$ and $x_0\in\R^N$ we denote by $B_R(x_0)$ the ball of radius $R$ and center $x_0$. Given a set $\Omega\subset\R^N$, we denote $\Omega^c:=\R^N\setminus\Omega$, and its characteristic function as $\chi_\Omega$. The space of the infinitely differentiable functions which are compactly supported is denoted by $C^\infty_0(\R^N)$, while $L^p(\R^N)$ with $p\in[1,+\infty]$ is the Lebesgue space of $p$-integrable functions. The norm of $L^p(\R^N)$ is denoted by $\|\cdot\|_p$. The space $\mathcal{S}$ is the Schwartz space of rapidly decreasing functions and $\mathcal{S}'$ the dual space of tempered distributions. For $q>0$ we define $\lfloor q\rfloor$ as the largest integer strictly less than $q$; if $q>1$ its conjugate H\"older exponent is $q':=\frac q{q-1}$. The symbol $\lesssim$ indicates that an inequality holds up to a multiplicative constant depending only on the structural constants. Finally, $o_n(1)$ denotes a vanishing real sequence as $n\to+\infty$. Hereafter, the letter $C$ will be used to denote positive constants which are independent of relevant quantities and whose value may change from line to line.

\subsection{Assumptions and main results}\label{Ass_Res}
As mentioned in the introduction, in \cite{dAC} a Trudinger-Moser inequality was proved for functions in $E$. This result is fundamental for us, since it allows to consider nonlinearities with exponential growth.

For $\alpha>0$ and $j_0\in\N$ we introduce the functions
\begin{equation}\label{Phi}
	\Phi_{\alpha,j_0}(t):=\e^{\alpha|t|^{\frac N{N-1}}}-\sum_{j=0}^{j_0-1}\frac{\alpha^j}{j!}|t|^{j\frac N{N-1}}\,.
\end{equation}
\begin{mainth}(\cite{dAC}, Theorem 1.6)\label{ThmAC_TM}
	Assume (A) and (Q) hold, and let $j_0:=\inf\big\{j\in\N\,|\,j\geq\frac{\gamma(N-1)}N\big\}$. Then, for each $u\in\Erad$ and $\alpha>0$, the function $Q(|\cdot|)\Phi_{\alpha,j_0}(u)$ belongs to $L^1(\R^N)$. Moreover, if
	$$0<\alpha<\widetilde\alpha_N:=\alpha_N\big(1+\tfrac{b_0}N\big)C_A^{1/(N-1)},$$
	where $\alpha_N:=N\omega_N^{1/(N-1)}$, with $\omega_{N-1}$ denoting the measure of the unit sphere in $\R^N$, and $C_A:=\inf_{x\in B_1(0)}A(|x|)$, then
	$$L(\alpha,A,Q):=\sup_{u\in\Erad,\|u\|\leq1}\intN Q(|x|)\Phi_{\alpha,j_0}(u)\dd x<+\infty\,.$$
\end{mainth}

Before stating the main result of this Section, let us introduce some additional conditions on $A$ and $Q$, which will be needed in estimating the mountain pass level.
\begin{enumerate}
	\item[(A')] There exists $r_0>0$ such that $A_0(1+|x|^\ell)\leq A(|x|)\leq A_0(1+|x|^L)$ for all $x\in B_{r_0}(0)$, with $A_0,\ell$ given by (A) and $L>0$;
	\item[(Q')] $\liminf_{r\to0^+}Q(r)/r^{\tb_0}=C_Q>0$ with $\tb_0:=\max\{b_0,b_0\left(1-\tfrac{\mu_0}{2N}\right)\}$;
	%\item[(\textit{f}$_3'$)] $F(t)\leq(1-\tau)tf(t)$ for any $t\geq0$ ;
\end{enumerate}
\textbf{Notation:} With a little abuse, from now on $A(x):=A(|x|)$ and similarly $Q(x):=Q(|x|)$.
\vskip0.2truecm
Concerning the nonlinearity $f$ we assume the following conditions:
\begin{enumerate}
	\item[(\textit{f}$_0$)] $f(t)>0$ for $t>0$ and $f(t)\equiv0$ for $t\leq0$;
	\item[(\textit{f}$_1$)] $f$ is a critical nonlinearity in the sense of Trudinger-Moser, namely there exists $\alpha_0>0$ such that
	\begin{equation*}%\label{crit_def}
		\lim_{t\to+\infty}\frac{f(t)}{\e^{\alpha t^{\frac N{N-1}}}}=\begin{cases}
			0&\quad\mbox{for}\ \ \alpha>\alpha_0,\\
			+\infty&\quad\mbox{for}\ \ \alpha<\alpha_0;\end{cases}
	\end{equation*}
	\item[(\textit{f}$_2$)] there exists $\tp>\gamma$ such that $f(t)=o(t^{\tp-1})$ as $t\to0^+$; %$\tp>\gamma\left(1-\tfrac\mu{2N}\right)$
	\item[(\textit{f}$_3$)] there exist $\tau\in\left(1-\tfrac2N,1\right)$ and $C>0$ such that
	$$\tau\leq\frac{F(t)f'(t)}{f^2(t)}\leq C\quad\mbox{for any}\ \,t>0\,;$$
	\item[(\textit{f}$_4$)] $\lim\limits_{t\to+\infty}\frac{F(t)f'(t)}{f^2(t)}=1$ or equivalently
	$\lim\limits_{t\to+\infty}\frac{\rm d}{{\rm d}t}\frac{F(t)}{f(t)}=0\,$;
	\item[(\textit{f}$_5$)] there exists $\beta>0$ and $\lambda\in\left(0,1+\tfrac N{N-1}\right]$ such that
	$$\lim_{t\to+\infty}\frac{t^\lambda f(t)F(t)}{\e^{2\alpha_0t^{\tfrac N{N-1}}}}\geq\beta>\beta_0\,,$$
	where $\beta_0=0$ if $\lambda<1+\tfrac N{N-1}$, while if $\lambda=1+\tfrac N{N-1}$ then $\beta_0>0$ is explicitly given in \eqref{beta0} and depends only on $N,A_0,L,r_0,C_Q,b_0,\alpha_0,\mu_0$.
\end{enumerate}

\begin{defn}[Solution of \eqref{Choq_log}]%\label{sol_Choquard_log}
	We say that $u\in E$ is a \textit{weak solution of} \eqref{Choq_log} if
	\begin{equation*}%\label{sol_Choquard_log_test}
		\intN A(x)|\nabla u|^{N-2}\nabla u\nabla\varphi\dd x=C_N\!\intN\!\!\left(\intN\log\frac1{|x-y|}Q(y)F(u(y))\dd y\!\right)\!Q(x)f(u(x))\varphi(x)\dd x
	\end{equation*}
	for all $\varphi\in E$.
\end{defn}

\begin{thm}\label{Thm_log}
	Assume conditions (A), (A'), (Q), (Q'), and that $f$ satisfies ($f_0$)-($f_5$). Then \eqref{Choq_log} has a positive radially symmetric weak solution in $\Erad$ such that
	\begin{equation}\label{logFF}
		\bigg|\intN\!\!\bigg(\log\frac1{|\cdot|}\ast QF(u)\bigg)QF(u)\dd x\bigg|<+\infty\,.
	\end{equation}
\end{thm}

Once we find a weak solution of the logarithmic Choquard equation \eqref{Choq_log}, we can go back to the original Schr\"odinger-Poisson system. First, we need a precise meaning of solution of \eqref{SP_0}.

For $s>0$ the weighted Lebesgue space $L_s(\R^N)$ is defined as
$$L_s(\R^N):=\Big\{u\in L^1_{loc}(\R^N)\,\Big|\,\int_{\R^N}\frac{|u(x)|}{1+|x|^{N+2s}}\dd x<+\infty\Big\}\,.$$
\begin{defn}\label{sol_Poisson}
	For $\ff\in\mathcal S'(\R^N)$ we say that a function $\phi\in L_{N/2}(\R^N)$ is a solution of the linear Poisson equation $(-\Delta)^{\frac N2}\phi=\ff$ in $\R^N$ if
	$$\intN\phi\,((-\Delta)^{\frac N2}\varphi)=\langle\ff,\varphi\rangle\qquad\mbox{for all}\ \,\varphi\in\mathcal S(\R^N)\,.$$
\end{defn}
\begin{defn}[Solution of \eqref{SP_0}]\label{sol_SP}
	We say that $(u,\phi)$ is a weak solution of the Schr\"odinger-Poisson system \eqref{SP_0} if
	\begin{equation*}
		\intN A(x)|\nabla u|^{N-2}\nabla u\nabla\varphi\dd x=\intN\phi\,Q(x)f(u)\varphi\dd x
	\end{equation*}
	for all $\varphi\in E$, and $\phi$ solves $(-\Delta)^{\frac N2}\phi=QF(u)$ in $\R^N$ in the sense of Definition \ref{sol_Poisson}.
\end{defn}

\begin{thm}[Existence for \eqref{SP_0}]\label{Thm_SP}
	Under the conditions of Theorem \ref{Thm_log} the Schr\"odinger-Poisson system \eqref{SP_0} possesses a solution $(u,\phi)\in E\times L_s(\R^N)$ for all $s>0$ such that:
	\begin{itemize}
		\item[i)] $u$ is positive, radially symmetric and \eqref{logFF} holds;
		\item[ii)] $\phi=\phi_u:=I_N\ast QF(u)$, with $I_N$ as in \eqref{Riesz_log}
	\end{itemize} 
\end{thm}

\begin{remark} Some comments are in order:
	\begin{enumerate}
		\item[\textit{i})] Our results may be seen equivalently as an extension to the zero-mass case of the analysis contained in \cite{BCT,CDL}, as well to the case of the system \eqref{SP_0} of those in \cite{dAC}. Moreover, compared to \cite{WCR,ChT,CSTW}, we admit in the coupling term a critical exponential nonlinearity.
		\item[\textit{ii})] The justification of the fact that from a solution of \eqref{Choq_log} one obtains a corresponding solution of \eqref{SP_0} is often neglected in the literature, by just advocating the reason that $\phi_u$ is the "natural" solution of the Poisson equation in \eqref{SP_0}. In fact, this can be made rigorous in a suitable setting, as we will do in Section \ref{Sec_SP}.
		\item[\textit{iii})] Since the weight $A$ is continuous and bounded below by (A), it is clear that for all $\Omega\subset\subset\R^N$ there exist constants $\underline a_0,\overline a_0>0$ such that $\underline a_0<A(x)<\overline a_0$ for all $x\in\Omega$. This implies that $E\subset D^{1,N}(\R^N)\subset W^{1,N}_{loc}(\R^N)$, where $D^{1,N}(\R^N)$ is the homogeneous Sobolev space, see \cite[Lemma II.6.1]{G}. Therefore, it is sufficient to prove the existence of a nonnegative solution of \eqref{Choq_log} in order to retrieve its positivity by the strong maximum principle for quasilinear equations, see \cite[Theorem 11.1]{PS}.
	\end{enumerate}
\end{remark}

\subsection{The approximating method}
As we mentioned in the Introduction, the applicability of variational methods to the logarithmic Choquard equation \eqref{Choq_log} is not straightforward. Indeed \eqref{Choq_log} has, at least formally, a variational structure related to the energy functional
\begin{equation*}%\label{eqn:log-fun}
	J(u):=\frac1N\intN\!A(x)|\nabla u|^N\!\dd x-\frac{C_N}2\intN\!\left(\intN\log\frac1{|x-y|}Q(y)F(u(y))\dd y\right)Q(x)F(u(x))\dd x.
\end{equation*}
However, $J$ is not well-defined on the natural Sobolev space $E$ because of the presence of the convolution term and the fact that the logarithm is unbounded both from below and from above. To overcome this difficulty, inspired by \cite{LRTZ}, see also \cite{CDL,CLR}, we will use an approximation technique as follows. Observing that
\begin{equation}\label{key_convergence}
	\log\frac1t=\lim_{\mu\to0^+}\frac{t^{-\mu}-1}\mu,
\end{equation}
set
$$G_\mu(x):=\frac{|x|^{-\mu}-1}\mu,\qquad\mu\in(0,1]\,,\ \ x\in\R^N,$$
and consider the approximating problem
\begin{equation}\label{Choq_approx}
	-\divv\big(A(x)|\nabla u|^{N-2}\nabla u\big)=C_N\left(G_\mu(\cdot)\ast Q(\cdot)F(u)\right)Q(x)f(u)\quad\ \mbox{in}\ \ \R^N,
\end{equation}
with corresponding functional
\begin{equation*}
	\begin{split}
		J_\mu(u):&=\frac1N\intN\!A(x)|\nabla u|^N+\frac{C_N}{2\mu}\bigg[\intN QF(u)\dd x\bigg]^2\\
		&\quad-\frac{C_N}{2\mu}\intN\intN\frac1{|x-y|^\mu}Q(x)F(u(x))Q(y)F(u(y))\dd x\dd y\\
		&=\frac1N\|u\|^N-\frac{C_N}2\intN\left(G_\mu(\cdot)\ast QF(u)\right)QF(u)\dd x\,.
	\end{split}
\end{equation*}
Unlike the logarithmic term in the original functional $J$, the power-type singularity in $G_\mu$ can be handled by the Hardy-Littlewood-Sobolev inequality (Lemma \ref{HLS}), and it is possible to prove under conditions ($f_0$)-($f_2$) that $J_\mu$ is well-defined and $C^1$ on $E$ with
\begin{equation}\label{J_der}
	\begin{split}
		J'_\mu(u)v=&\intN\!A(x)|\nabla u|^{N-2}\nabla u\nabla v\dd x-C_N\intN(G_\mu(\cdot)\ast QF(u))Qf(u)v\dd x\,,
	\end{split}
\end{equation}
provided conditions (A) and (Q) are fulfilled, as we will see next (Lemma \ref{J_welldefined}). The overall strategy consists then in producing a critical point $u_\mu\in E$ for $J_\mu$ for all $\mu\in(0,\mu_0)$ for some $\mu_0>0$, and then pass to the limit as $\mu\to0$ in order to obtain a critical point $u_0$ for the original functional $J$ in $E$, which a posteriori satisfies $\left(\log\frac1{|\cdot|}\ast QF(u)\right)QF(u)\in L^1(\R^N)$. In this way, on the one hand we do not need to restrict to a constraint or weighted subspace of the natural space $E$, but of course, on the other hand, we have to be careful in our estimates, namely they must be independent of $\mu$ so that we are allowed to pass to the limit at the end.

Note also that, since the functions $A$ and $Q$ are radially symmetric, and in order to retrieve compactness in light of Theorem \ref{Thm_cpt_emb}, we will work in the radial subspace $\Erad$. This is not a restriction since, by means of a suitable version of Palais' principle of symmetric criticality in \cite[Lemma 4.1]{dAC}, any critical point $u_0$ for $J$ in $\Erad$ is also a critical point in the whole $E$, namely a weak solution of \eqref{Choq_log}.

\section{Preliminary results}\label{Prel}

Let us first point out some immediate consequences of the assumptions ($f_0$)-($f_5$):
\begin{remark}\label{Rmk_ass}
	\begin{enumerate}
		\item[(i)] From ($f_1$), ($f_2$), and \eqref{Phi}, it is easy to infer that for fixed $\alpha>\alpha_0$, $p\geq1$ and for any $\varepsilon>0$ one has
		\begin{equation}\label{f-C-above}
			|f(t)|\leq\varepsilon|t|^{\tp-1}+C_1(\alpha,p,\varepsilon)|t|^{p-1}\Phi_{\alpha,j_0}(t),\qquad t\in\R,
		\end{equation}
		for some $C_1(\alpha,p,\varepsilon)>0$, and consequently, %we obtain easily that for fixed $\alpha>\alpha_0$ and for any $\varepsilon>0$,
		\begin{equation}\label{F-C-above}
			|F(t)|\leq\varepsilon|t|^\tp+C_2(\alpha,p,\varepsilon)|t|^p\Phi_{\alpha,j_0}(t),\qquad t\in\R,
		\end{equation}
		for some $C_2(\alpha,p,\varepsilon)>0$.
		\item[(ii)] Assumption ($f_3$) implies that $f$ is monotone increasing. Moreover,
		$$\frac{{\rm d}}{{\rm d}t}\frac{F(t)}{f(t)}=\frac{f^2(t)-F(t)f'(t)}{f^2(t)}\leq 1-\tau\,,$$
		from which one infers
		$$F(t)\leq(1-\tau)tf(t)\quad\ \mbox{for any}\ \ t\geq0\,.$$
		\item[(iii)] From ($f_4$) one may deduce that there exists $M_0>0$ and $s_0>0$ such that % for any $\varepsilon>0$ there exists $M_\varepsilon$ such that
		\begin{equation}\label{Rmk_ass_f7}
			F(t)\leq M_0f(t)\quad\ \mbox{for any}\ \ t\geq s_0\,.
		\end{equation}
		see e.g. \cite[p.2 (1.3)]{CDL}.
		\item[(iv)] ($f_5$) is related to the well-known de Figueiredo-Miyagaki-Ruf condition \cite{dFMR} and is crucial in order to estimate the mountain pass level and gain compactness, see Lemma \ref{MP_level}. We note here that unlike most of the references in the literature \cite{AF,ChT,CMR,AFS,BM1,dAC,BM2}, we do not prescribe a global bound from below on the growth of $f$ of the kind \eqref{global_ass}, but a condition at infinity, where $\beta>0$ is prescribed large only in the specific case of $\lambda=1+\tfrac N{N-1}$. A similar condition, but just in the limit case for $\lambda$, appears also in \cite{ACTY,CT,BCT}, see also \cite{CSTW}. We also point out that the lower bound $\beta_0$ in \eqref{beta0} is explicit and depends just on structural constants.
	\end{enumerate}
\end{remark}

\noindent In the sequel we will use some useful estimates which we collect in the next two lemmas. 
\begin{lemma}\label{Lem_basic_est}
	Let $\mu\in(0,1]$. Then,
	$$\frac{t^{-\mu}-1}{\mu}\geq\log\frac1t\qquad\mbox{for all}\ \ t\in(0,1]\,.$$
	Moreover, for all $\nu>\mu$ there exists $C_\nu>0$ such that
	$$
	\frac{t^{-\mu}-1}{\mu}\leq C_\nu t^{-\nu}\qquad\mbox{for all}\ \ t>0\,.
	$$
\end{lemma}

\begin{lem}[Lemma 2.3, \cite{LY}]\label{estimate_Sani}
	Let $\alpha>0$ and $r>1$. Then for any $\beta>r$ there exists a constant $C_\beta>0$ such that
	\begin{equation*}
		(\Phi_{\alpha,j_0}(t))^r\leq C_\beta\Phi_{\alpha\beta,j_0}(t)\qquad\mbox{for all}\ \,t>0.
	\end{equation*}
\end{lem}

We end this section by recalling the well-known Hardy-Littlewood-Sobolev inequality, see \cite[Theorem 4.3]{LL}, which will be frequently used throughout the paper, and a version of the radial Lemma which, thanks to assumption (A), is suitable in our space $E$, see \cite[Lemma 2.3]{dAC}.
\begin{lemma}(Hardy-Littlewood-Sobolev inequality)\label{HLS}
	Let $N\geq1$, $q,r>1$, and $\mu\in(0,N)$ with $\tfrac1q+\tfrac\mu N+\tfrac1r=2$. There exists a constant $C=C(N,\mu,q,r)$ such that for all $f\in L^q(\R^N)$ and $h\in L^r(\R^N)$ one has
	$$
	\int_{\R^N}\left(\frac1{|\cdot|^\mu}\ast f\right)\!h\dd x\leq C\|f\|_q\|h\|_r\,.
	$$
\end{lemma}

\begin{lem}(Radial Lemma)\label{RadialLemma}
	Let $N\geq2$, $\ell>0$. For any $r_0>0$ there exists $C=C(N,\ell)>0$ such that
	$$|u(x)|\leq C\left(\int_{\R^N\setminus B_{r_0}(0)}|x|^\ell|\nabla u|^N\dd x\right)^\frac1N|x|^{-\frac\ell{N-1}}$$
	for all $u\in C^1_{0,\rad}(\R^N)$ and $|x|\geq r_0$.
\end{lem}

\section{The Choquard equation with logarithmic Riesz kernel}\label{Sec_Riesz_log}
\subsection{The approximating problem}

We start by showing that the functional $J_\mu$ which is naturally associated to the approximating problem \eqref{Choq_approx} is well-defined in the natural space $E$.
\begin{lem}\label{J_welldefined}
	Under assumptions (A), (Q), and ($f_0$)-($f_2$) the exists $\mu_0$ such that for all $\mu\in(0,\mu_0)$ the functional $J_\mu:\Erad\to\R$ is well-defined and $C^1$.
\end{lem}
\begin{proof}
	We need to check the finiteness the two terms which originate from the nonlinearity $f$. First, by \eqref{F-C-above} and the H\"older inequality with $q,q'>1$, one has
	\begin{equation*}
		\begin{split}
			\intN QF(u)&\les\intN Q|u|^\tp+\left(\intN Q|u|^{pq'}\right)^\frac1{q'}\left(\intN Q\,\Phi_{\alpha r,j_0}(u)\right)^\frac1q,
		\end{split}
	\end{equation*}
	with $r>q$ by Lemma \ref{estimate_Sani}. Since assumption (Q) is fulfilled, the last term is finite for all $q>0$ in view of Theorem \ref{ThmAC_TM}, while the continuous embedding given by Theorem \ref{Thm_cpt_emb} ensures that the first two terms are finite too, provided one chooses $q$ sufficiently small so that $pq'>\gamma$.
	On the other hand, by means of the Hardy-Littlewood-Sobolev inequality (Lemma \ref{HLS}) and the H\"older inequality one gets
	\begin{equation}\label{RHS_HLS}
		\begin{split}
			\bigg|&\intN\left(\frac1{|\cdot|^\mu}\ast QF(u)\right)QF(u)\bigg|\les\left(\intN|QF(u)|^\frac{2N}{2N-\mu}\right)^\frac{2N-\mu}N\\
			&\les\left[\intN Q^\frac{2N}{2N-\mu}|u|^\frac{2N\tp}{2N-\mu}+\left(\intN Q^\frac{2N}{2N-\mu}|u|^\frac{2Npq'}{2N-\mu}\right)^\frac1{q'}\left(\intN Q^\frac{2N}{2N-\mu}\Phi_{r\alpha,j_0}(u)\right)^\frac1q\right]^\frac{2N-\mu}N\!\!\!,
		\end{split}
	\end{equation}
	with $r>\tfrac{2N q}{2N-\mu}$ again by Lemma \ref{estimate_Sani}. We claim that there exists $\mu_0>0$ such that for all $\mu\in(0,\mu_0)$ there exist exponents $\nu,\nu_0>-N$ possibly depending on $\mu$, such that the function $Q^\frac{2N}{2N-\mu}$ satisfies condition (Q), namely such that
	\begin{equation*}%\label{Q_nu_nu0}
		\limsup_{r\to0^+}\frac{Q^\frac{2N}{2N-\mu}(r)}{r^{\nu_0}}<+\infty\quad\mbox{and}\quad\limsup_{r\to+\infty}\frac{Q^\frac{2N}{2N-\mu}(r)}{r^\nu}<+\infty\,.
	\end{equation*}
	Indeed, since $\frac{2N}{2N-\mu}>0$, this is equivalent to the condition
	\begin{equation*}%\label{Q_nu_nu0_equiv}
		\limsup_{r\to0^+}\frac{Q(r)}{r^{\nu_0\left(1-\frac\mu{2N}\right)}}<+\infty\quad\mbox{and}\quad\limsup_{r\to+\infty}\frac{Q(r)}{r^\nu\left(1-\frac\mu{2N}\right)}<+\infty\,.
	\end{equation*}
	Since $Q$ verifies (Q) with exponents $b_0,b>-N$, this in turns amounts to find $\nu_0,\nu>-N$, possibly depending on $\mu$ such that
	\begin{equation*}%\label{}
		\exists\lim_{r\to0^+}r^{b_0\left(1+\tfrac\mu{2N-\mu}\right)-\nu_0}<+\infty\quad\mbox{and}\quad\exists\lim_{r\to+\infty}r^{b\left(1+\tfrac\mu{2N-\mu}\right)-\nu}<+\infty,
	\end{equation*}
	that is
	\begin{equation}\label{limits}
		b_0\left(1+\tfrac\mu{2N-\mu}\right)\geq\nu_0\quad\mbox{and}\quad b\left(1+\tfrac\mu{2N-\mu}\right)\leq\nu.
	\end{equation}
	Defining then
	\begin{equation}\label{nu_nu0}
		\nu_0:=\begin{cases}
			b_0\ \  &\mbox{if}\ \,b_0\geq0,\\
			b_0\left(1+\frac\mu{2N-\mu}\right) & \mbox{if}\ \,b_0<0,\\
		\end{cases}
		\quad\mbox{and}\quad
		\nu:=\begin{cases}
			b\left(1+\frac\mu{2N-\mu}\right)\ \  &\mbox{if}\ \,b\geq0,\\
			b & \mbox{if}\ \,b<0,\\
		\end{cases}
	\end{equation}
	it is easy to verify that \eqref{limits} holds true, that $\nu>-N$, and also $\nu_0>-N$, provided $\mu_0$ is small enough. This is clear if $b_0\geq0$, while, if $b_0\in(-N,0)$, there exists $\varepsilon_0>0$ such that $b_0-\varepsilon_0>-N$ and therefore, choosing $\mu_0:=2N\varepsilon_0(|b_0|+\varepsilon_0)^{-1}$, one has that $\nu_0=b_0-|b_0|\mu(2N-\mu)^{-1}\geq b_0-\varepsilon_0>-N$ for all $\mu\in(0,\mu_0)$. All in all, we can conclude that $Q^\frac{2N}{2N-\mu}$ satisfies condition (Q) with exponents $\nu,\nu_0$ defined in \eqref{nu_nu0}. Therefore, thanks again to Theorems \ref{Thm_cpt_emb} and \ref{ThmAC_TM}, the terms in \eqref{RHS_HLS} are all finite, provided again $q$ is sufficiently small, since the condition for the first term $\tp>\gamma\left(1-\frac\mu{2N}\right)$ is automatically satisfied by ($f_2$). Finally, by adapting the standard arguments in \cite[Lemma 3.1]{CDL} to our weighted setting, one gets also that $J_\mu\in C^1(\Erad,\R)$ with derivative as in \eqref{J_der}.
\end{proof}

Henceforth, we always assume conditions (A) and (Q), and consider $\mu\in(0,\mu_0)$, with $\mu_0$ given by Lemma \ref{J_welldefined}.

Next we prove that for such $\mu$ the functional $J_\mu$ enjoys a mountain-pass geometry.
\begin{lemma}\label{MP_geom}
	Let $\mu\in(0,\mu_0)$ and assume ($f_0$)-($f_3$). Then, there exist constants $\rho,\eta>0$ and $e\in\Erad$ such that:
	\begin{enumerate}
		\item[{\rm(i)}] $J_\mu|_{S_\rho}\geq\eta>0$, where $S_\rho=\big\{u\in\Erad\,|\,\|u\|=\rho\big\}$;
		\item[{\rm(ii)}] $\|e\|>\rho$ and $J_\mu(e)<0$.
	\end{enumerate}
\end{lemma}
\begin{proof} (i) By Lemma \ref{Lem_basic_est} with $\nu=\mu_0$, one has
	\begin{equation}\label{MP_1}
		\begin{split}
			J_\mu(u)&\geq\frac{\|u\|^N}N-\frac{C_N}2\iint_{\{|x-y|\leq1\}}\frac{|x-y|^{-\mu}-1}\mu\,Q(x)F(u(x))Q(y)F(u(y))\dd x\dd y\\
			&\geq\frac{\|u\|^N}N-\frac{C_{\mu_0}C_N}2\iint_{\{|x-y|\leq1\}}\frac{Q(x)F(u(x))Q(y)F(u(y))}{|x-y|^{\mu_0}}\dd x\dd y\,.
		\end{split}
	\end{equation}
	Applying Lemma \ref{HLS} as in \eqref{RHS_HLS} on the second term, which we denote by $T$, we get
	\begin{equation*}%\label{T}
		T\les\left[\intN Q^\frac{2N}{2N-\mu_0}|u|^\frac{2N\tp}{2N-\mu_0}+\left(\intN Q^\frac{2N}{2N-\mu_0}|u|^\frac{2Npq'}{2N-\mu_0}\right)^\frac1{q'}\left(\intN Q^\frac{2N}{2N-\mu_0}\Phi_{\alpha r\|u\|^{\frac N{N-1}},j_0}\left(\frac u{\|u\|}\right)\right)^\frac1q\right]^\frac{2N-\mu_0}N\!\!\!,
	\end{equation*}
	with $r>\frac{2N q}{2N-\mu}$. Prescribing $\rho^{\frac N{N-1}}<(2\alpha_0N)^{-1}(\omega_NC_A)^{1/(N-1)}(N+b_0)(2N-\mu_0)$, then the third term is uniformly bounded by Theorem \ref{ThmAC_TM}. Therefore, defining $\tQ:=Q^\frac{2N}{2N-\mu_0}$, from \eqref{MP_1} we infer
	\begin{equation*}
		J_\mu(u)\ges\|u\|^N-\|u\|^{2\tp}_{L^\frac{2N\tp}{2N-\mu_0}_\tQ}-\|u\|^{2p}_{L^\frac{2Npq'}{2N-\mu_0}_\tQ}\ges\|u\|^N-\|u\|^{2\tp}-\|u\|^{2p},
	\end{equation*}
	having applied once more the continuous embedding in weighted Lebesgue spaces by Theorem \ref{Thm_cpt_emb}, which is permitted by the condition $\tp>\gamma$ and the choice of $q$ small enough. Since one can choose $p>N$ and since $\tp>\gamma\geq N$ by definition, condition (i) holds true provided $\rho_0$ is sufficiently small.
	
	(ii) Take $\varphi_0\in C^\infty_{0,\rad}(B_\frac14(0))$ with $\varphi_0\equiv1$ in $B_\frac18(0)$ and $|\nabla\varphi_0|\leq C$. For $t>0$ set
	\begin{equation*}%\label{eqn:AR-0}
		\Psi(t):=\frac12\bigg(\intN QF(t\varphi_0)\dd x\bigg)^2.
	\end{equation*}
	Then
	$$\Psi'(t)=\frac1t\left(\intN Qf(t\varphi_0)t\varphi_0\right)\left(\intN QF(t\varphi_0)\right).$$
	By Remark \ref{Rmk_ass}(ii), one has then
	\begin{equation*}
		\frac{\Psi(t)}{\Psi'(t)}\geq\frac2{(1-\tau)t}.
	\end{equation*}
	Integrating this on $[1,s]$ one infers
	\begin{equation}\label{MP_2}
		\frac12\bigg(\intN QF(s\varphi_0)\dd x\bigg)^2=\Psi(s)\geq\Psi(1)s^\frac2{1-\tau}=\frac12\bigg(\intN QF(\varphi_0)\dd x\bigg)^2s^\frac2{1-\tau}.
	\end{equation}
	By the compact support of $\varphi_0$, Lemma \ref{Lem_basic_est}, and \eqref{MP_2}, one has
	\begin{equation*}
		\begin{split}
			J_\mu(t\varphi_0)&=\frac{t^N}N\|\varphi_0\|^N-\frac{C_N}2\iint_{\{|x-y|\leq\frac12\}}\frac{|x-y|^{-\mu}-1}\mu\, Q(x)F(t\varphi_0(x))Q(y)F(t\varphi_0(y))\dd x\dd y\\
			&\leq\frac{t^N}N\|\varphi_0\|^N-\frac{C_N}2\iint_{\{|x-y|\leq\frac12\}}\log\frac1{|x-y|}\, Q(x)F(t\varphi_0(x))Q(y)F(t\varphi_0(y))\dd x\dd y\\
			&\leq\frac{t^N}N\|\varphi_0\|^N-\frac{C_N}2\log2\left(\intN QF(t\varphi_0)\right)^2\\
			&\leq\frac{t^N}N\|\varphi_0\|^N-\frac{C_N}2\log2\left(\intN QF(\varphi_0)\right)^2t^\frac2{1-\tau}.
		\end{split}
	\end{equation*}
	Since $\tau>1-\frac2N$ by ($f_3$), one has $J_\mu(t\varphi_0)\to-\infty$ as $t\to+\infty$, and (ii) follows by taking $e=t\varphi_0$ with $t$ large enough.
\end{proof}
As a consequence of Lemma \ref{MP_geom}, the mountain pass level
$$c_\mu:=\inf_{\gamma\in\Gamma}\max_{t\in[0,1]}J_\mu(\gamma(t))\,,$$
where
$$\Gamma:=\left\{\gamma\in C([0,1],E)\,|\,\gamma(0)=0,\gamma(1)=e\right\}$$
is well-defined. Moreover, from the Ekeland Variational Principle, the mountain pass geometry yields the existence of a Cerami sequence at level $c_\mu$ for any fixed $\mu\in(0,\mu_0)$, see e.g. \cite{Ekeland}, namely, there exists $(u_n^\mu)_n\subset\Erad$ such that
$$J_{\mu}(u_n^\mu)\to c_\mu\quad\text{and}\quad(1+\|u_n^\mu\|)J'_\mu(u_n^\mu)\to0\quad\text{in}\,\,E'$$
as $n\to+\infty$. For the sake of a lighter notation, hereafter $u_n:=u_n^\mu$. In details,
\begin{equation}\label{Cerami_J}
	J_\mu(u_n)=\frac1N\intN\!A(x)|\nabla u_n|^N-\frac{C_N}2\intN\left(G_\mu(\cdot)\ast QF(u_n)\right)QF(u_n)=c_\mu+o_n(1),
\end{equation}
and for all $\varphi\in E$ one has
\begin{equation}\label{Cerami_Jder}
	J_\mu'(u_n)[\varphi]=\intN\!A(x)|\nabla u_n|^{N-2}\nabla u_n\nabla\varphi-C_N\intN\left(G_\mu(\cdot)\ast QF(u_n)\right)Qf(u_n)\varphi=o_n(1)\|\varphi\|\,,
\end{equation}
from which,
\begin{equation}\label{Cerami_Jder_un}
	J_\mu'(u_n)[u_n]=\intN\!A(x)|\nabla u_n|^N-C_N\intN\left(G_\mu(\cdot)\ast QF(u_n)\right)Qf(u_n)u_n=o_n(1)\|u_n\|\,.
\end{equation}

\begin{remark}\label{rem_cMP_unif}
	Observe from the proof of Lemma \ref{MP_geom} that there exist two constants $\underline c,\overline c>0$ independent of $\mu$ such that $\underline c<c_\mu<\overline c$.
\end{remark}

\begin{lemma}\label{Lem:c-bounded}
	Assume that ($f_0$)-($f_3$) hold. Let $(u_n)_n\subset\Erad$ be a Cerami sequence of $J_\mu$ at level $c_\mu$. Then $(u_n)_n$ is bounded in $E$ and
	\begin{equation}\label{Lem:estimates_GFF_GFfu}
		\bigg|\intN[G_\mu(x)\ast QF(u_n)]QF(u_n)\dd x\bigg|<C\,,\quad\ \bigg|\intN[G_\mu(x)\ast QF(u_n)]Qf(u_n)u_n\dd x\bigg|<C\,.
	\end{equation}
\end{lemma}
\begin{proof}
	In order to infer the boundedness of the Cerami sequence $(u_n)_n$ at level $c_\mu$, we argue as in \cite[Lemma 3.5]{CDL}. We introduce the sequence
	\begin{equation}
		v_n:=\begin{cases}
			\frac{F(u_n)}{f(u_n)}\ \  &\mbox{if}\ \,u_n>0,\\
			(1-\tau)u_n & \mbox{if}\ \,u_n<0.\\
		\end{cases}
	\end{equation}
	for which, by Remark \ref{Rmk_ass}(ii), $|v_n|\leq(1-\tau)|u_n|$ hold, hence $(v_n)_n\subset\Erad$. Using $\varphi=v_n$ in \eqref{Cerami_Jder} one gets
	\begin{equation*}
		\begin{split}
			&o_n(1)\|v_n\|=(1-\tau)\int_{\{u_n<0\}}A(x)|\nabla u_n|^N+\int_{\{u_n>0\}}A(x)|\nabla u_n|^{N-2}\nabla u_n\nabla\left(\frac{F(u_n)}{f(u_n)}\right)\\
			&-(1-\tau)C_N\int_{\{u_n<0\}}\left(G_\mu(\cdot)\ast QF(u_n)\right)Qf(u_n)u_n-C_N\int_{\{u_n>0\}}\left(G_\mu(\cdot)\ast QF(u_n)\right)QF(u_n)\,.
		\end{split}
	\end{equation*}
	Since $f\equiv0$ on $\R^-$ the third term vanishes; moreover,
	\begin{equation*}
		\nabla\left(\frac{F(u_n)}{f(u_n)}\right)=\left(1-\frac{F(u_n)f'(u_n)}{(f(u_n))^2}\right)\nabla u_n\,,
	\end{equation*}
	hence by \eqref{Cerami_J} one obtains
	\begin{equation}\label{vn_eq}
		\begin{split}
			(1-\tau)&\int_{\{u_n<0\}}A(x)|\nabla u_n|^N+\int_{\{u_n>0\}}A(x)\left(1-\frac{F(u_n)f'(u_n)}{(f(u_n))^2}\right)|\nabla u_n|^N\\
			&=C_N\intN\left(G_\mu(\cdot)\ast QF(u_n)\right)QF(u_n)+o_n(1)\|v_n\|\\
			&=\frac2N\intN\!A(x)|\nabla u_n|^N-2c_\mu+o_n(1)\|v_n\|\,.
		\end{split}
	\end{equation}
	Using ($f_3$), from this one infers
	\begin{equation}\label{uniform_bound_mu}
		\left(\tau-\left(1-\tfrac2N\right)\right)\|u_n\|^N\leq2c_\mu+o(1),
	\end{equation}
	that is, the boundedness of $(u_n)_n$ in $E$. Note also that this bound is uniform with respect to $\mu$, see Remark \ref{rem_cMP_unif}. The bounds in \eqref{Lem:estimates_GFF_GFfu} are then consequences of this uniform bound and \eqref{Cerami_J} and \eqref{Cerami_Jder_un}.
\end{proof}

\begin{remark}\label{Rmk_nonneg_Cerami}
	Thanks to the uniform boundedness of Cerami sequences of Lemma \ref{Lem:c-bounded}, from now on we can always suppose that Cerami sequences at level $c_\mu$ are nonnegative. Indeed, $u_n^-:=\min\{u_n,0\}\in\Erad$ and $u_n^-\leq0$ and thus, recalling that $f\equiv0$ on $\R^-$ by ($f_0$), one has
	\begin{equation*}
		\begin{split}
			\|u_n^-\|^N&=\|u_n^-\|^N-\frac{C_N}{2\mu}\left[\left(\intN QF(u_n)\right)\left(\intN Qf(u_n)u_n^-\right)-\intN\left(\frac1{|\cdot|^\mu}\ast QF(u_n)\right)Qf(u_n)u_n^-\right]\\
			&=J_\mu'(u_n)[u_n^-]\leq\|J_\mu'(u_n)\|_{E'}\|u_n^-\|=o_n(1)
		\end{split}
	\end{equation*}
	since $\|u_n^-\|\leq\|u_n\|\leq C$ by Lemma \ref{Lem:c-bounded}. This implies that $u_n^-\to0$ in $E$ as $n\to+\infty$ and therefore that $(u_n^+)_n$ is a Cerami sequence of $J_\mu$ at level $c_\mu$, which we will simply denote by $(u_n)_n$.
\end{remark}

By the reflexivity\footnote{The reflexivity of $E$ can be shown in the usual way thanks to the reflexivity of the weighted Lebesgue spaces $L^N(\R^N,A(\cdot)\dd x)$ for $N\geq2$ see e.g. \cite{DS}.} of $E$, the uniform bound proved in Lemma \ref{Lem:c-bounded} yields the existence of $u_\mu\in\Erad$ such that $u_n\rightharpoonup u_\mu$ in $E$ as $n\to+\infty$. By the compact embedding in Theorem \ref{Thm_cpt_emb}, this implies $u_n\to u_\mu$ in $L^t_Q(\R^N)$ provided $Q$ fulfills condition (Q) and $t>\gamma$, and hence also $u_n\to u_\mu$ a.e. in $\R^N$. In order to prove that $u_\mu$ is a critical point of $J_\mu$, one needs to show that
\begin{equation}\label{conv_NL}
	\intN\left(G_\mu(\cdot)\ast QF(u_n)\right)QF(u_n)\to\intN\left(G_\mu(\cdot)\ast QF(u_\mu)\right)QF(u_\mu)\quad\ \mbox{as}\ \,n\to+\infty\,.
\end{equation}
To this aim, we need to obtain a uniform integrability result for $F$ in some weighted Lebesgue spaces, see the next Lemma \ref{Lem:integral-F}. A preliminary important step is a careful estimate of the mountain pass level $c_\mu$ for each $\mu\in(0,\mu_0)$ by a quantity which does not depend on $\mu$. In order to obtain finer estimates, here we need to assume the additional conditions (A') and (Q').

\begin{lem}\label{MP_level}
	Under (A'), (Q'), ($f_0$)-($f_3$), and ($f_5$), one has 
	\begin{equation}\label{MP_est}
		\sup_{\mu\in(0,\mu_0)}c_\mu<\frac{\omega_{N-1}A_0}N\left(\frac{\tb_0+N}{\alpha_0}\right)^{N-1}.
	\end{equation}
\end{lem}
\begin{proof}
	Let $\rho\leq r_0$ given by (A') and consider the Moser sequence
	\begin{equation*}
		\tw_n(x):=\begin{cases}
			(\log n)^{1-\frac1N}\,,\ \  &\mbox{if}\ \,0\leq|x|\leq\tfrac\rho n,\\
			\tfrac{\log\tfrac\rho{|x|}}{(\log n)^\frac1N} & \mbox{if}\ \,\tfrac\rho n<|x|<\rho,\\
			0 & \mbox{if}\ \,|x|\geq\rho.\\
		\end{cases}
	\end{equation*}
	Using (A') we estimate from below its norm in $E$ as
	\begin{equation*}%\label{Moser_norm_below}
		\begin{split}
			\intN\!A(x)|\nabla\tw_n|^N&=\frac{\omega_{N-1}}{\log n}\int_{\tfrac\rho n}^n\frac{A(r)}r\dd r\geq\frac{\omega_{N-1}A_0}{\log n}\int_{\tfrac\rho n}^n\frac{1+r^\ell}r\dd r\\
			&=\omega_{N-1}A_0\left(1+\frac{\rho^\ell}{\ell\log n}+o\left(\frac1{\log n}\right)\right),%=:A_0(1+\delta_n)
		\end{split}
	\end{equation*}
	and analogously from above, hence we can state that
	\begin{equation}\label{Moser_norm}
		\|\tw_n\|^N=\omega_{N-1}A_0(1+\delta_n), \quad\ \ \mbox{with}\quad\ \frac{\rho^\ell/\ell}{\log n}+o\left(\frac1{\log n}\right)\leq\delta_n\leq\frac{\rho^L/L}{\log n}+o\left(\frac1{\log n}\right).
	\end{equation}
	Hence defining
	$$w_n:=\frac{\tw_n}{\left[\omega_{N-1}A_0(1+\delta_n)\right]^\frac1N},$$
	one has $\|w_n\|=1$ for all $n\in\N$. To prove \eqref{MP_est}, it is sufficient to find a suitable $\CyrB>0$ independent of $\mu\in(0,\mu_0)$ such that there exists $n_0\in\N$ such that
	\begin{equation}\label{Be}
		\max_{t\geq0}J_\mu(tw_{n_0})<\CyrB.
	\end{equation}
	Suppose by contradiction that \eqref{Be} does not hold. This means that for all $n\in\N$ there exists $t_n>0$ such that
	\begin{equation*}%\label{Be_not}
		J_\mu(t_nw_n)=\max_{t\geq0}J_\mu(tw_n)\geq\CyrB,
	\end{equation*}
	hence $\tfrac\dd{\dd t}\big|_{t=t_n}J_\mu(tw_n)=0$. From these conditions, recalling $\|w_n\|=1$, one infers
	\begin{equation}\label{MP_est_1a}
		\frac{t_n^N}N\geq\frac{C_N}2\intN\left(G_\mu(\cdot)\ast QF(t_nw_n)\right)QF(t_nw_n)+\CyrB
	\end{equation}
	and
	\begin{equation}\label{MP_est_1b}
		t_n^N=C_N\intN\left(G_\mu(\cdot)\ast QF(t_nw_n)\right)Qf(t_nw_n)t_nw_n\,.
	\end{equation}
	Choosing $\rho\leq\frac14$, by Lemma \ref{Lem_basic_est} one has $G_\mu(|x-y|)\geq\frac1{|x-y|}\geq\log2$, and therefore \eqref{MP_est_1a} yields
	\begin{equation}\label{MP_est_fact0}
		t_n^N\geq N\CyrB\,.
	\end{equation}
	We aim now to prove that $\limsup_{n\to+\infty}t_n^N\leq N\CyrB$. Suppose by contradiction instead that
	\begin{equation}\label{MP_est_fact1}
		t_n^N\geq N\CyrB+\delta_0
	\end{equation} for $\delta_0>0$ and $n$ sufficiently large. Then,
	\begin{equation*}
		\begin{split}
			B:&=\intN\intN\frac{|x-y|^{-\mu}-1}\mu\,Q(y)F(t_nw_n(y))Q(x)f(t_nw_n(x))t_nw_n(x)\dd x\dd y\\
			&\geq\int_{B_\frac\rho n(0)}\int_{B_\frac\rho n(0)}\log\frac1{|x-y|}\,Q(y)F(t_nw_n(y))Q(x)f(t_nw_n(x))t_nw_n(x)\dd x\dd y\\
			&\geq\log\frac n{2\rho}\left(\int_{B_\frac\rho n(0)}QF(t_nw_n)\right)\left(\int_{B_\frac\rho n(0)}Qf(t_nw_n)t_nw_n\right)\\
			&\geq\log\frac n{2\rho}\left(\int_{B_\frac\rho n(0)}Q\sqrt{F(t_nw_n)f(t_nw_n)t_nw_n}\right)^2,
		\end{split}
	\end{equation*}
	where in the last step we used the Cauchy-Schwarz inequality. Note that $w_n$ is constant in $B_\frac\rho n(0)$. By ($f_5$) we then deduce
	\begin{equation*}
		\begin{split}
			B&\geq\frac\beta2\log\frac n{2\rho}\left(\int_{B_\frac\rho n(0)}Q(x)(t_nw_n)^\frac{1-\lambda}2\e^{\alpha_0(t_nw_n)^{\frac N{N-1}}}\dd x\right)^2\\
			&=\frac\beta2\log\frac n{2\rho}\left(\int_{B_\frac\rho n(0)}Q\right)^2t_n^{1-\lambda}\left(\frac{(\log n)^{N-1}}{\omega_{N-1}A_0(1+\delta_n)}\right)^\frac{1-\lambda}N\e^{2\alpha_0t_n^{\frac N{N-1}}\log n\left[\omega_{N-1}A_0(1+\delta_n)\right]^{-\frac1{N-1}}}.
		\end{split}
	\end{equation*}
	By (Q'), recalling $\tb_0:=\max\{b_0,b_0\left(1-\tfrac{\mu_0}{2N}\right)\}$, we can estimate from below
	\begin{equation*}
		\int_{B_\frac\rho n(0)}Q\geq C_Q\int_{B_\frac\rho n(0)}|x|^{\tb_0}\dd x=\omega_{N-1}C_Q\int_0^{\frac\rho n}r^{\tb_0+N-1}\dd r=\frac{\omega_{N-1}C_Q}{\tb_0+N}\left(\frac\rho n\right)^{\tb_0+N},
	\end{equation*}
	provided $\rho$ is small enough, say $\rho<r_Q$, from which we infer
	\begin{equation*}
		\begin{split}
			B&\geq\frac\beta2\log\frac n{2\rho}\,\rho^{2\tb_0+2N}\frac{\omega_{N-1}^2C_Q^2}{(\tb_0+N)^2}t_n^{1-\lambda}\left(\frac{(\log n)^{N-1}}{\omega_{N-1}A_0(1+\delta_n)}\right)^\frac{1-\lambda}N\\
			&\quad\cdot\exp{\left(\frac{2\alpha_0t_n^{\frac N{N-1}}}{(\omega_{N-1}A_0)^\frac1{N-1}(1+\delta_n)^\frac1{N-1}}-(2\tb_0+2N)\right)\log n}\,.
		\end{split}
	\end{equation*}
	Combining this with \eqref{MP_est_1b} one obtains
	\begin{equation}\label{MP_est_boundbelowtn}
		\begin{split}
			t_n^{N-1+\lambda}&\geq C_N\frac\beta2\log\frac n{2\rho}\rho^{2\tb_0+2N}\frac{\omega_{N-1}^2C_Q^2}{(\tb_0+N)^2}\left(\frac{(\log n)^{N-1}}{\omega_{N-1}A_0(1+\delta_n)}\right)^\frac{1-\lambda}N\\
			&\quad\cdot\exp\left\{2\left(\frac{\alpha_0t_n^{\frac N{N-1}}}{(\omega_{N-1}A_0)^\frac1{N-1}(1+\delta_n)^\frac1{N-1}}-(\tb_0+N)\right)\log n\right\}.
		\end{split}
	\end{equation}
	In order to avoid a contradiction, it must be
	$$\frac{\alpha_0t_n^{\frac N{N-1}}}{(\omega_{N-1}A_0)^\frac1{N-1}(1+\delta_n)^\frac1{N-1}}-(\tb_0+N)\leq0,$$
	that is,
	\begin{equation}\label{MP_est_fact2}
		t_n^N\leq\omega_{N-1}A_0(1+\delta_n)\left(\frac{\tb_0+N}{\alpha_0}\right)^{N-1}.
	\end{equation}
	Comparing \eqref{MP_est_fact1} and \eqref{MP_est_fact2}, and since $\delta_n=o_n(1)$ as $n\to+\infty$, we see that choosing
	\begin{equation}\label{Be_def}
		\CyrB:=\frac{\omega_{N-1}A_0}N\left(\frac{\tb_0+N}{\alpha_0}\right)^{N-1},
	\end{equation}
	one reaches the claimed contradiction, namely one gets
	\begin{equation*}%\label{MP_est_limsup_tn}
		\limsup_{n\to+\infty}t_n\leq\omega_{N-1}A_0\left(\frac{\tb_0+N}{\alpha_0}\right)^{N-1}\!,
	\end{equation*}
	which in particular implies that $(t_n)_n$ is bounded, and from \eqref{MP_est_fact1} that
	\begin{equation}\label{MP_est_lim_tn}
		\exists\lim_{n\to+\infty}t_n=\omega_{N-1}A_0\left(\frac{\tb_0+N}{\alpha_0}\right)^{N-1}\!.
	\end{equation}
	
	For the sake of a compact notation, let us now define
	\begin{equation}\label{xin_K}
		\xi_n:=\frac{\alpha_0t_n^{\frac N{N-1}}}{(\tb_0+N)(\omega_{N-1}A_0)^\frac1{N-1}}\quad\mbox{and}\quad K_\rho:=\frac{C_N\rho^{2\tb_0+2N}\omega_{N-1}^2C_Q^2}{2(\omega_{N-1}A_0)^{\tfrac{1-\lambda}N}(\tb_0+N)^2}
	\end{equation}
	so that \eqref{MP_est_boundbelowtn} can be written as
	\begin{equation*}%\label{MP_est_boundbelowtn2}
		\begin{split}
			t_n^{N-1+\lambda}&\geq\frac{\beta K_\rho}{(1+\delta_n)^\frac{1-\lambda}N}\log\frac n{2\rho}\left(\log n\right)^{\frac{N-1}N(1-\lambda)}\exp\left\{2(\tb_0+N)\log n\left(\frac{\xi_n}{(1+\delta_n)^\frac1{N-1}}-1\right)\right\}.
		\end{split}
	\end{equation*}
	Since $\delta_n=o_n(1)$, as $n\to+\infty$, one has
	\begin{equation}\label{MP_est_xi_n}
		\frac{\xi_n}{(1+\delta_n)^{\frac1{N-1}}}-1=\xi_n-1-\frac{\delta_n}{N-1}+o(\delta_n)\geq-\frac{\delta_n}{N-1}+o(\delta_n),
	\end{equation}
	by means of \eqref{MP_est_fact0} and \eqref{Be_def}. Combining \eqref{MP_est_xi_n} with the upper bound for $\delta_n$ in \eqref{Moser_norm}, and fixing $\rho=\min\{\tfrac14,r_0,r_Q\}$, yields
	\begin{equation}\label{MP_est_fact3}
		t_n^{N-1+\lambda}\geq\beta K_\rho\e^{-2\frac{(\tb_0+N)\rho^L}{(N-1)L}+o_n(1)}\log\frac n{2\rho}\left(\log n\right)^{\frac{N-1}N(1-\lambda)}(1+\delta_n)^{\frac{\lambda-1}N}.
	\end{equation}
	Since $t_n$ is bounded, we then reach a contradiction, which would prove \eqref{Be} with \eqref{Be_def},  in the following two cases:
	\begin{enumerate}
		\item[\text{Case 1:}] $\lim_{n\to+\infty}\log\frac n{2\rho}\left(\log n\right)^{\frac{N-1}N(1-\lambda)}=+\infty$, and this irrespective of $\beta>0$;
		\item[\text{Case 2:}] $\lim_{n\to+\infty}\log\frac n{2\rho}\left(\log n\right)^{\frac{N-1}N(1-\lambda)}\in\R^+$, provided $\beta>\beta_0$ large enough.
	\end{enumerate}
	\paragraph{Case 1} The above limit is true in case $\tfrac{(N-1)(1-\lambda)}N+1>0$, namely provided $\lambda<1+\frac N{N-1}$, which is indeed one of the two possibilities in assumption ($f_5$).
	\paragraph{Case 2} This happens when $\lambda=1+\frac N{N-1}$, since in this case
	$$\log\frac n{2\rho}\left(\log n\right)^{\frac{N-1}N(1-\lambda)}=1-\frac{\log(2\rho)}{\log n}.$$
	Combining \eqref{MP_est_fact3} with \eqref{MP_est_lim_tn}, we then see that
	\begin{equation*}
		\begin{split}
			(N\CyrB)^\frac{N-1+\lambda}N&=\lim_{n\to+\infty}t_n^{N-1+\lambda}\geq\lim_{n\to+\infty}\beta K_\rho\e^{-2\frac{(\tb_0+N)\rho^L}{(N-1)L}+o_n(1)}\left(1-\frac{\log(2\rho)}{\log n}\right)\left(1+O\left(\frac1{\log n}\right)\right)^\frac{\lambda-1}N\\
			&=\beta K_\rho\e^{-2\frac{(\tb_0+N)\rho^L}{(N-1)L}}.
		\end{split}
	\end{equation*}
	The contradiction is reached once we impose in ($f_5$) that $\beta>\beta_0$ so that 
	\begin{equation*}
		(N\CyrB)^\frac{N-1+\lambda}N=\beta_0 K_\rho\e^{-2\frac{(\tb_0+N)\rho^L}{(N-1)L}},
	\end{equation*}
	that is, taking into account \eqref{Be_def} and \eqref{xin_K},
	\begin{equation}\label{beta0}
		\beta>\beta_0:=\frac{A_0}{\omega_{N-1}}\frac{2(\tb_0+N)^{N+2}}{\alpha_0^NC_NC_Q^2\rho^{2(\tb_0+N)}}\e^{\frac{2(\tb_0+N)\rho^L}{(N-1)L}}.
	\end{equation}
	Hence, all in all, 
	\begin{equation*}%\label{Be_final}
		\sup_{\mu\in(0,\mu_0)}\max_{t\geq0}J_\mu(tw_{n_0})<\left(\frac{(\tb_0+N)}{\alpha_0}\right)^{N-1}\frac{\omega_{N-1}A_0}N,
	\end{equation*}
	which in turns implies \eqref{MP_est} by definition of $c_\mu$.
\end{proof}

Thanks to the fine uniform estimate of the mountain pass level $c_\mu$ by Lemma \ref{MP_level}, we can now get a uniform integrability result for $(F(u_n))_n$, which is a key step in order to gain compactness and be able to prove that the limit $u_\mu$ is indeed a weak solution of the Choquard equation.

Henceforth, unless otherwise stated, all our assumptions will be always taken into account, namely (A), (A'), (Q), (Q'), and ($f_0$)--($f_5$).

\begin{lem}\label{Lem:integral-F}
	Let $(u_n)_n$ be a Cerami sequence of $J_\mu$ at level $c_\mu$. Then there exist $C>0$, $\nu_*>1$, and $\kappa_0>1$, independent of $n$ and $\mu\in(0,\mu_0)$ (up to a smaller $\mu_0$), such that
	$$\intN Q^\nu f(u_n)u_n\,\dd x\leq C \quad \text { and } \quad \intN Q^\nu F(u_n)^\kappa\dd x\leq C\,.$$
	for any $\nu\in[1,\nu_*)$ and $\kappa\in[1,\kappa_0)$.
\end{lem}
\begin{proof}
	First we note that, for a weight $Q$ satisfying (Q), then $Q^\nu$ fulfills (Q) for small values of $\nu>1$. Indeed, $Q^\nu$ is readily continuous and positive, and moreover it is easy to see that
	\begin{equation*}
		\limsup_{r\to0^+}\frac{Q^\nu(r)}{r^{\epsilon_0}}=\left(\limsup_{r\to0^+}\frac{Q(r)}{r^{b_0}}r^{b_0-\frac{\epsilon_0}\nu}\right)^\nu<+\infty,
	\end{equation*}
	provided $b_0-\frac{\epsilon_0}\nu\geq 0$, and similarly
	\begin{equation*}
		\limsup_{r\to+\infty}\frac{Q^\nu(r)}{r^\epsilon}=\left(\limsup_{r\to+\infty}\frac{Q(r)}{r^b}r^{b-\frac\epsilon\nu}\right)^\nu<+\infty,
	\end{equation*}
	provided $b-\frac{\epsilon_0}\nu\geq 0$. Taking $\epsilon_0:=b_0\nu$ and $\epsilon:=b\nu$, then $\epsilon_0,\epsilon>-N$ provided $\nu<\nu_*:=\min\{\tfrac Nb,\tfrac N{b_0}\}$.
	
	We show next the uniform boundedness of $\intN Q^\nu F(u_n)^\kappa\dd x$. Inspired by \cite{BCT}, let us introduce the auxiliary function
	\begin{equation}\label{H}
		H(t):=\int_0^t\sqrt[N]{\frac N2\frac{F(s)f'(s)}{(f(s))^2}-\frac{N-2}2}\dd s\quad\mbox{for}\ \,t \geq 0\,,
	\end{equation}
	and define $v_n:=H\left(u_n\right)$. Then
	\begin{equation*}%\label{norm_vn}
		\begin{split}
			\|v_n\|^N&=\intN\!A(x)|H'(u_n)\nabla u_n|^N=\intN\!A(x)\left(\frac N2\frac{F(u_n)f'(u_n)}{(f(u_n))^2}-\frac N2+1\right)|\nabla u_n|^N.\\
			&
		\end{split}
	\end{equation*}
	Recalling \eqref{vn_eq} and that one can choose $u_n\geq0$ by Remark \ref{Rmk_nonneg_Cerami}, we infer
	\begin{equation}\label{vn_mu}
		\|v_n\|^N=Nc_\mu+o_n(1)\,.
	\end{equation}
	By Lemma \ref{MP_level} this in particular implies that $(v_n)_n\subset E$ is uniformly bounded. Thanks to ($f_4$), for any $\varepsilon>0$ there exists $s_\varepsilon>0$ such that
	$$1-\varepsilon\leq\left(\frac N2\left(\frac{F(s)f'(s)}{(f(s))^2}-1\right)+1\right)^\frac1N\leq 1+\varepsilon\quad\ \mbox{for all}\ \,s\geq s_\varepsilon\,.$$
	Hence, by ($f_3$) we can estimate as follows:
	\begin{equation*}
		\begin{split}
			v_n&=H(u_n)=\int_0^{s_\varepsilon}\left(\frac N2\frac{F(s)f'(s)}{(f(s))^2}-\frac N2+1\right)^\frac1N\!\dd s+\int_{s_\varepsilon}^{u_n}\left(\frac N2\frac{F(s)f'(s)}{(f(s))^2}-\frac N2+1\right)^\frac1N\!\dd s\\
			&\geq\left(\frac N2(\tau-1)+1\right)^\frac1Ns_\varepsilon+(1-\varepsilon)(u_n-s_\varepsilon)\geq(1-\varepsilon)(u_n-s_\varepsilon),
		\end{split}
	\end{equation*}
	by the choice of $\tau$ in ($f_3$). We conclude that
	\begin{equation}\label{un_vn}
		u_n\leq s_\varepsilon+\frac{v_n}{1-\varepsilon}\,.
	\end{equation}
	We are now ready to estimate
	\begin{equation*}
		\intN Q^\nu F(u_n)^\kappa\dd x=\int_{\{u_n\leq s_\varepsilon\}}Q^\nu F(u_n)^\kappa\dd x+\int_{\{u_n>s_\varepsilon\}}Q^\nu F(u_n)^\kappa\dd x\,.
	\end{equation*}
	Since $u_n$ is uniformly bounded in the first term, we can estimate
	\begin{equation}\label{QFu_small_values}
		\int_{\{u_n\leq s_\varepsilon\}}Q^\nu F(u_n)^\kappa\dd x\leq C_\varepsilon\intN Q^\nu|u_n|^{\tp\kappa}\leq C_\varepsilon\|u_n\|^{\tp\kappa}\leq C
	\end{equation}
	by using the continuous embedding $E\hookrightarrow L_{Q^\nu}^{\tp\kappa}(\R^N)$ since $Q^\nu$ verifies (Q), and $\tp>\gamma$ and $\kappa>1$, see Theorem \ref{Thm_cpt_emb}, and then Lemma \ref{Lem:c-bounded}. On the other hand, by \eqref{un_vn}, for $\alpha>\alpha_0$ and $p>1$ one gets
	\begin{equation}\label{QFu_big_values}
		\begin{split}
			\int_{\{u_n>s_\varepsilon\}}Q^\nu F(u_n)^\kappa\dd x&\leq C_\varepsilon\int_{\{u_n>s_\varepsilon\}}Q^\nu|u_n|^{p\kappa}\Phi_{\kappa\alpha,j_0}(u_n)\\
			&\leq C_{\varepsilon,p}\int_{\{u_n>s_\varepsilon\}}Q^\nu\Phi_{\kappa\alpha(1+\varepsilon),j_0}(u_n)\\
			&\leq C_{\varepsilon,p}\int_{\{u_n>s_\varepsilon\}} Q^\nu\Phi_{\kappa\alpha(1+\varepsilon),j_0}\left(s_\varepsilon+\frac{v_n}{1-\varepsilon}\right)\\
			&\leq C_{\varepsilon,p}\intN Q^\nu\Phi_{\kappa\alpha_\varepsilon\|v_n\|^{\frac N{N-1}},j_0}\left(\frac{v_n}{\|v_n\|}\right)
		\end{split}
	\end{equation}
	where $\alpha_\varepsilon:=\alpha\left(\frac{1+\varepsilon}{1-\varepsilon}\right)^{\frac N{N-1}}$. The last inequality follows from the elementary estimate $1+\varepsilon<(1+\varepsilon)^{\frac N{N-1}}$ and the observation that for large values of $u_n$, also $v_n$ is large so that $\left( s_\varepsilon+\frac{v_n}{1-\varepsilon}\right)^{\frac N{N-1}}\sim\left(\frac{v_n}{1-\varepsilon}\right)^{\frac N{N-1}}$.
	
	In order to apply Theorem \ref{ThmAC_TM} (taking again into account that $Q^\nu$ verifies (Q)) and get a uniform estimate on the right-hand side of \eqref{QFu_big_values}, one needs
	\begin{equation}\label{kappa_TM}
		\kappa\alpha_\varepsilon\|v_n\|^{\frac N{N-1}}\leq(b_0+N)(\omega_{N-1}A_0)^\frac1{N-1}.
	\end{equation}
	By \eqref{vn_mu}, this is verified provided
	\begin{equation*}
		\kappa\frac{\alpha_\varepsilon}{\alpha_0}\leq\frac{(b_0+N)(\omega_{N-1}A_0)^\frac1{N-1}}{\alpha_0\left(N\sup_{\mu\in(0,\mu_0)}c_\mu\right)^\frac1{N-1}}.
	\end{equation*}
	By the estimate on $c_\mu$ in Lemma \ref{MP_level}, it follows
	\begin{equation*}
		\frac{(b_0+N)(\omega_{N-1}A_0)^\frac1{N-1}}{\alpha_0\left(N\sup_{\mu\in(0,\mu_0)}c_\mu\right)^\frac1{N-1}}>\frac{b_0+N}{\tb_0+N}\,.
	\end{equation*}
	In turns, one can apply Theorem \ref{ThmAC_TM} provided
	\begin{equation*}
		\kappa\,\frac{\tb_0+N}{b_0+N}\frac{\alpha_\varepsilon}{\alpha_0}\leq\Lambda:=\frac{(b_0+N)(\omega_{N-1}A_0)^\frac1{N-1}}{\alpha_0\left(N\sup_{\mu\in(0,\mu_0)}c_\mu\right)^\frac1{N-1}}\,.
	\end{equation*}
	Notice however that
	$$\frac{\alpha_\varepsilon}{\alpha_0}=\frac\alpha{\alpha_0}\left(\frac{1+\varepsilon}{1-\varepsilon}\right)^\frac N{N-1}\searrow1\quad\ \mbox{as}\ \,\alpha\searrow\alpha_0\ \,\mbox{and}\ \,\varepsilon\to0\,$$
	and by definition of $\tb_0$ in (Q'), $\tfrac{\tb_0+N}{b_0+N}=1$ if $b_0\geq0$, while otherwise
	$$\frac{\tb_0+N}{b_0+N}=\frac{b_0\left(1-\frac{\mu_0}{2N}\right)+N}{b_0+N}=1-\frac{\mu_0}{2N\left(1+\frac N{b_0}\right)}\to1$$
	as $\mu_0\to0$. Hence, fixed $\delta$ small, one can always choose $\mu_0>0$, $\alpha>\alpha_0$, $\varepsilon>0$ such that
	$$\frac{\tb_0+N}{b_0+N}\,\frac\alpha{\alpha_0}\,\left(\frac{1+\varepsilon}{1-\varepsilon}\right)^\frac N{N-1}<1+\delta\,.$$
	Therefore, choosing $\delta>0$ small such that $\kappa_0:=\frac\Lambda{1+\delta}>1$, then \eqref{kappa_TM} holds and hence, together with \eqref{QFu_small_values} and \eqref{QFu_big_values}, the uniform bound $\intN Q^\nu F(u_n)^\kappa\leq C$ for all $n\in\N$ is proved.
	
	Finally, the estimate $\intN Q^\nu f(u_n)u_n\,\dd x\leq C$ follows by a similar but easier argument.
\end{proof}

In view of Lemma \ref{Lem:integral-F}, and following \cite{LRTZ,CDL,CLR} which ultimately rely on the dominated convergence theorem, we aim now to prove the convergence \eqref{conv_NL}. We split this proof in a number of technical lemmas. 

\begin{lem}\label{Lem:conv_Riesz_mu}
	Let $\mu\in(0,\mu_0)$ and $(u_n)_n$ be a bounded Cerami sequence for $J_\mu$ at level $c_\mu$, which is weakly converging to $u_\mu$ in $E$. Then 
	\begin{equation}\label{QF_conv}
		\intN QF(u_n)\dd x\to\intN QF(u_\mu)\dd x
	\end{equation}
	and
	\begin{equation}\label{QF_conv_Riesz_mu}
		\intN\left(\frac1{|\cdot|^\mu}\ast QF(u_n)\right)QF(u_n)\to \intN\left(\frac1{|\cdot|^\mu}\ast QF(u_\mu)\right)QF(u_\mu).
	\end{equation}
\end{lem}

\begin{proof}
	Using the mean value theorem, there exists $\tau_n(x)\in(0,1)$ such that
	\begin{equation*}%\label{estimate_QF_differenza}
		\begin{split}
			\intN&Q|F(u_n)-F(u_\mu)|\dd x
			=\intN Q|f\left(u_\mu+\tau_n(x)(u_n-u_\mu)\right)(u_n-u_\mu)|\dd x\\
			&\les\intN Q(|u_n|^{\tp-1}+|u_\mu|^{\tp-1})|u_n-u_\mu|\dd x\\
			&\quad+\intN Q(u_n+u_\mu)^{p-1}\Phi_{\alpha,j_0}\big(u_\mu+\tau_n(x)(u_n-u_\mu)\big)|u_n-u_\mu|\dd x\\
			&\les\left(\|u_n\|_{L_Q^\tp(\R^N)}+\|u_\mu\|_{L_Q^\tp(\R^N)}\right)\|u_n-u_\mu\|_{L_Q^\tp(\R^N)}\\
			&\quad+\int_{\{u_n>u_\mu\}}Q|u_n|^{p-1}\Phi_{\alpha,j_0}(u_n)|u_n-u_\mu|+\int_{\{u_n\leq u_\mu\}}Q|u_\mu|^{p-1}\Phi_{\alpha,j_0}(u_\mu)|u_n-u_\mu|+o_n(1)\\
			&=:o_n(1)+S_1+S_2,
		\end{split}
	\end{equation*}
	since $u_n\to u_\mu$ in $L^\tp_Q(\R^N)$ by compact embedding. Let us estimate the remaining two terms separately. First,
	\begin{equation*}%\label{S_2}
		S_2\les\left(\intN Q|u_\mu|^{(p-1)\theta'\eta}\right)^\frac1{\theta'\eta}\left(\intN Q\,\Phi_{\alpha r,j_0}(u_\mu)\right)^\frac1{\theta'\eta'}\left(\intN Q|u_n-u_\mu|^\theta\right)^\frac1\theta=o_n(1),
	\end{equation*}
	with $r>\theta'\eta'$, by choosing $\theta>\gamma$ and $(p-1)\theta'\eta>\gamma$ and applying Theorem \ref{ThmAC_TM}. Recalling \eqref{un_vn}, and again choosing $r>\theta'\eta'$, we similarly estimate
	\begin{equation}\label{S_1}
		S_1\les\left(\intN Q|u_n|^{(p-1)\theta'\eta}\right)^\frac1{\theta'\eta}\left(\intN Q\,\Phi_{\alpha r,j_0}\left(s_\varepsilon+\frac{v_n}{1-\varepsilon}\right)\right)^\frac1{\theta'\eta'}\left(\intN Q|u_n-u_\mu|^\theta\right)^\frac1\theta.
	\end{equation}
	Arguing as in the proof of Lemma \ref{Lem:integral-F}, we can prove the uniform boundedness of the second term in \eqref{S_1}. This implies that $S_1=o_n(1)$ which, combined with \eqref{S_1}, yields \eqref{QF_conv}.
	
	Let us now prove \eqref{QF_conv_Riesz_mu}. First, by $J_\mu'(u_n)[u_n]=o_n(1)$ and Lemmas \ref{Lem:c-bounded} and \ref{Lem:integral-F} one infers
	\begin{equation*}
		\begin{split}
			\frac{C_N}\mu&\intN\left(\frac1{|\cdot|^\mu}\ast QF(u_n)\right)Qf(u_n)u_n\\
			&=\|u_n\|^N+\frac{C_N}\mu\left(\intN QF(u_n)\right)\left(\intN Qf(u_n)u_n\right)-2c_\mu+o_n(1)\leq C\,.
		\end{split}
	\end{equation*}
	Hence, using \eqref{Rmk_ass_f7}, for any $\varepsilon>0$ there exists $M_\varepsilon>0$ such that
	\begin{equation}\label{Claim_1}
		\int_{\{u_n\geq M_\varepsilon\}}\left(\frac1{|\cdot|^\mu}\ast QF(u_n)\right)QF(u_n)\leq\frac{M_0}{M_\varepsilon}\int_{\{u_n\geq M_\varepsilon\}}\left(\frac1{|\cdot|^\mu}\ast QF(u_n)\right)Qf(u_n)u_n<\varepsilon\,.
	\end{equation}
	Moreover, $u_\mu\in E$ implies $\left(\tfrac1{|\cdot|^\mu}\ast QF(u_\mu)\right)QF(u_\mu)\in L^1(\R^N)$, see \eqref{RHS_HLS}, hence 
	\begin{equation}\label{Claim_2}
		\int_{\{u_\mu\geq M_\varepsilon\}}\left(\frac1{|\cdot|^\mu}\ast QF(u_\mu)\right)QF(u_\mu)<\varepsilon\,.
	\end{equation}
	We claim now that there exists $C>0$ independent of $n$ and $\mu\in(0,\mu_0)$ such that for all $x\in\R^N$ one has
	\begin{equation}\label{CLAIM}
		\intN\frac{Q(y)F(u_n(y))}{|x-y|^\mu}\dd y\leq C\,.
	\end{equation}
	Note indeed, that
	\begin{equation*}%\label{CLAIM_big}
		\int_{\{|x-y|\geq1\}}\frac{Q(y)F(u_n(y))}{|x-y|^\mu}\dd y\leq\intN QF(u_n)\leq C
	\end{equation*}
	by Lemma \ref{Lem:integral-F}, and, by H\"older's inequality, that
	\begin{equation}\label{CLAIM_small}
		\int_{\{|x-y|\leq1\}}\frac{Q(y)F(u_n(y))}{|x-y|^\mu}\dd y\leq\left(\int_{\{|x-y|\leq1\}}\frac1{|x-y|^{\mu q}}\dd y\right)^\frac1q\left(\intN Q^{q'}F(u_n)^{q'}\right)^\frac1{q'}.
	\end{equation}
	Choosing $q$ sufficiently large so that $q'\in(1,\min\{\kappa_0,\nu_*\})$, the second term is bounded, again by Lemma \eqref{Lem:integral-F}; then, it is sufficient to choose $\mu_0$ sufficiently small so that $|\cdot|^{\mu q}\in L^1(B_1(0))$ for all $\mu\in(0,\mu_0)$, and the first term is also bounded independently of $\mu$. Hence \eqref{CLAIM} holds.
	
	Take now $C_\varepsilon>0$ such that by Remark \ref{Rmk_ass}(i) one has $F(u_n)\chi_{\{|u_n|<M_\varepsilon\}}<\varepsilon|u_n|^\tp+C_\varepsilon|u_n|^q$, and define
	$$G(x,u_n):=\left(\frac1{|\cdot|^\mu}\ast QF(u_n)\right)Q(\varepsilon|u_n|^\tp+C_\varepsilon|u_n|^p).$$
	Then we have
	\begin{equation*}
		\begin{split}
			\Big|\intN&(G(x,u_n)-G(x,u_\mu))\dd x\Big|\leq\varepsilon\int\left(\frac1{|\cdot|^\mu}\ast QF(u_n)\right)Q|u_n|^\tp+\varepsilon\intN\left(\frac1{|\cdot|^\mu}\ast QF(u_\mu)\right)Q|u_\mu|^\tp\\
			&\quad+C_\varepsilon\intN\left(\frac1{|\cdot|^\mu}\ast QF(u_\mu)\right)\left||u_n|^p-|u_\mu|^p\right|+C_\varepsilon\intN\left(\frac1{|\cdot|^\mu}\ast Q\left(F(u_n)-F(u_\mu)\right)\right)Q|u_\mu|^p\\
			&\leq 2C\varepsilon+C_\varepsilon\intN Q|u_n-u_\mu|^p+C_\varepsilon\intN\left(\frac1{|\cdot|^\mu}\ast Q\left(F(u_n)-F(u_\mu)\right)\right)Q|u_\mu|^p,
		\end{split}
	\end{equation*}
	where \eqref{CLAIM} is used in the last inequality. Note that the second term is $o_n(1)$ by the compact embedding $E\hookrightarrow L_Q^p(\R^N)$, since $p>\gamma$. Moreover, by \eqref{CLAIM} we infer
	\begin{equation*}
		\begin{split}
			\intN\bigg(\frac1{|\cdot|^\mu}\ast Q&\left(F(u_n)-F(u_\mu)\right)\bigg)Q|u_\mu|^p\\
			&\leq\intN\!\left(\intN\frac{Q(x)|u_\mu|^p(x)}{|x-y|^\mu}\dd x\right)Q(y)\left|F(u_n(y))-F(u_\mu(y))\right|\dd y\\
			&\leq C\intN Q\left|F(u_n)-F(u_\mu)\right|=o_n(1)
		\end{split}
	\end{equation*}
	by \eqref{CLAIM} and \eqref{QF_conv}. Hence, using the Lebesgue dominated convergence theorem, from
	$$\intN G(x,u_n)\dd x\to\intN G(x,u_\mu)\dd x$$
	we obtain
	$$\intN\left(\frac1{|\cdot|^\mu}\ast QF(u_n)\right)QF(u_n)\chi_{\{|u_n|<M_\varepsilon\}}\to\intN\left(\frac1{|\cdot|^\mu}\ast QF(u_\mu)\right)QF(u_\mu)\chi_{\{|u_\mu|<M_\varepsilon\}},$$
	which, together with \eqref{Claim_1} and \eqref{Claim_2} concludes the proof.
\end{proof}

\begin{prop}\label{Existence_u_mu}
	Let $(u_n)_n$ be a bounded Cerami sequence for $J_\mu$ at level $c_\mu$. Then, there exists a nontrivial $u_\mu\in\Erad$ such that $u_n\to u_\mu$ in $E$ as $n\to+\infty$, and $u_\mu$ is a critical point for $J_\mu$.
\end{prop}
\begin{proof}
	The uniform boundedness given by Lemma \ref{Lem:c-bounded} ensures the existence of $u_\mu\in\Erad$ such that $u_n\rightharpoonup u_\mu$ as $n\to+\infty$. In order to prove that $u_\mu$ is a critical point for $J_\mu$, we first claim that
	\begin{equation}\label{conv_Riesz_phi}
		\intN\left(\frac1{|\cdot|^\mu}\ast QF(u_n)\right)Qf(u_n)\varphi\to\intN\left(\frac1{|\cdot|^\mu}\ast QF(u_\mu)\right)Qf(u_\mu)\varphi
	\end{equation}
	for all $\varphi\in C^\infty_0(\R^N)$. Note that $u_n\to u_\mu$ in $L^1_{Q,loc}(\R^N)$; moreover, the function
	\begin{equation*}
		g(x,u_n(x)):=\left(\intN\frac1{|x-y|^\mu}Q(y)F(u_n(y))\dd y\right)Q(x)f(u_n(x))\varphi(x)
	\end{equation*}
	is uniformly bounded in $L^1(\R^N)$ by means of \eqref{CLAIM} and the fact that $\intN Qf(u_n)\varphi\leq C$ for all $n\in\N$ arguing as in Lemma \ref{Lem:integral-F}. Analogously one has $g(\cdot,u_\mu(\cdot))\in L^1(\R^N)$. Hence, \eqref{conv_Riesz_phi} follows by \cite[Lemma 2.1]{dFMR}. By \eqref{QF_conv} and the convergence $f(u_n)\to f(u_\mu)$ in $L^1_{Q,loc}(\R^N)$ which is similarly proved, we also have
	\begin{equation}\label{conv_QFQf}
		\left(\intN QF(u_n)\right)\left(\intN Qf(u_n)\varphi\right)\to\left(\intN QF(u_\mu)\right)\left(\intN Qf(u_\mu)\varphi\right).
	\end{equation}
	Combining \eqref{conv_Riesz_phi} and \eqref{conv_QFQf} one gets
	\begin{equation*}%\label{conv_Gmu}
		\intN\left(G_\mu(\cdot)\ast QF(u_n)\right)Qf(u_n)\varphi\to\intN\left(G_\mu(\cdot)\ast QF(u_\mu)\right)Qf(u_\mu)\varphi\,,
	\end{equation*}
	from which, together with $u_n\rightharpoonup u_\mu$ in $E$, it follows that $J_\mu'(u_\mu)=0$. In order to prove that $u_\mu\neq0$, let us show that $u_n\to u_\mu$ in $E$. We have
	\begin{equation}\label{conv_norm}
		 \begin{split}
		 	o_n(1)&=\left(J_\mu'(u_n)-J_\mu'(u_\mu)\right)(u_n-u_\mu)=\intN\!A(x)\left(|\nabla u_n|^{N-2}\nabla u_n-|\nabla u_\mu|^{N-2}\nabla u_\mu\right)\left(\nabla u_n-\nabla u_\mu\right)\\
		 	&\quad+\frac1\mu\left[\intN QF(u_n)\intN Qf(u_n)(u_n-u_\mu)-\intN QF(u_\mu)\intN Qf(u_\mu)(u_n-u_\mu)\right]\\
		 	&\quad-\frac1\mu\intN\left[\left(\frac1{|\cdot|^\mu}\ast QF(u_n)\right)Qf(u_n)-\left(\frac1{|\cdot|^\mu}\ast QF(u_\mu)\right)Qf(u_\mu)\right](u_n-u_\mu)\\
		 	&\geq c_N\int A(x)|\nabla(u_n-u_\mu)|^N+\frac1\mu\left(\cR_1(u_n,u_\mu)+\cR_2(u_n,u_\mu)\right),
		 \end{split}
	\end{equation}
	where in the first term we used the well-known inequality
	$$(|a|^{p-2}a-|b|^{p-2}b)(a-b)\geq C_p|a-b|^p,$$
	which holds for all $p\geq2$ and $a,b\in\R^N$, see e.g. \cite{Simon}. Arguing as for \eqref{QF_conv}, we infer that $f(u_n)(u_n-u_\mu)\to 0$ and $f(u_\mu)(u_n-u_\mu)\to 0$ in $L^1_Q(\R^N)$ as $n\to+\infty$. This together with \eqref{QF_conv} implies $\cR_1(u_n,u_\mu)=o_n(1)$. Recalling \eqref{CLAIM}, similarly one can also prove that $\cR_2(u_n,u_\mu)=o_n(1)$. All in all, from \eqref{conv_norm} one infers that $u_n\to u_\mu$ in $E$. By the continuity of $J_\mu$ one then infers $J_\mu(u_n)\to J_\mu(u_\mu)$, which yields $J_\mu(u_\mu)=c_\mu\geq\underline c$ by Remark \ref{rem_cMP_unif}. This readily implies that $u_\mu\neq0$.
\end{proof}

\subsection{Existence for the Choquard equation \eqref{Choq_log}: Proof of Theorem \ref{Thm_log}}

Let $(u_\mu)_\mu$ be the family of critical points of the family of approximating functional $(J_\mu)_\mu$ given by Proposition \ref{Existence_u_mu}. Note that $u_\mu\geq0$ as a consequence of Remark \ref{Rmk_nonneg_Cerami}. Combining \eqref{uniform_bound_mu} with Lemma \ref{MP_level} and Proposition \ref{Existence_u_mu}, we deduce that
\begin{equation}\label{unif_bdd_u_mu}
	\|u_\mu\|\leq C
\end{equation}
where $C$ is independent of $\mu$. Hence, up to a subsequence, there exists $u_0\in\Erad$ such that
\begin{equation}\label{convergences}
	\aligned
		u_\mu\rightharpoonup u_0\ \quad&\text{in}\ E,\\
		u_\mu\to u_0\ \quad&\text{in}\,\, L^t_Q(\R^N),\quad\mbox{for all}\ \ t>\gamma,\\
		u_\mu\to u_0\ \quad&\text{a.e. in}\,\, \R^N,
	\endaligned
\end{equation}
as $\mu\to0$ by Theorem \ref{Thm_cpt_emb}. From $J_\mu'(u_\mu)=0$, i.e.
\begin{equation}\label{umu_crit_pt}
	\intN\!A(x)|\nabla u_\mu|^{N-2}\nabla u_\mu\nabla\varphi=\intN\!\left(\intN\frac{|x-y|^\mu-1}\mu Q(y)F(u_\mu(y))\dd y\right)Q(x)f(u_\mu(x))\varphi(x)\dd x\,,
\end{equation}
and thanks to the pointwise convergence \eqref{key_convergence}, we expect that the limit function $u_0$ satisfies $J'(u_0)=0$ via the dominated convergence theorem. It is clear that the left-hand side of \eqref{umu_crit_pt} converges to the respective with $u_0$ by weak convergence. Concerning the right-hand side, we need as usual to split the convolution term with respect to $|x-y|\gtreqqless1$. For large values of $|x-y|$ we need to rely on a uniform decay of $(u_\mu)_\mu$ at infinity, while we need an improved version of \eqref{CLAIM_small} to deal with small values of $|x-y|$.

Let us start by considering the case $|x-y|\leq1$.
\begin{lemma}\label{Lem:omega}
	For any $\omega\in\left[1,\min\{\nu_*,k_0\}\right)$ there exists a constant $C>0$ independent of $\mu$ such that
	\begin{equation*}
		\int_{\{|x-y|\leq1\}}\frac{Q(y)F(u_\mu(y))}{|x-y|^\frac{4(\omega-1)}{3\omega}}\dd y\leq C
	\end{equation*}
	for all $x\in\R^N$.
\end{lemma}
\begin{proof}
	By H\"older's inequality,
	\begin{equation*}%\label{F_omega}
		\int_{\{|x-y|\leq 1\}}\!\frac{Q(y)F(u_\mu(y))}{|x-y|^{\frac{4(\omega-1)}{3\omega}}}\dd y\leq\left(\int_{\{|x-y|\leq1\}}\frac{\dd y}{|x-y|^\frac43}\right)^{\!\!\frac{\omega-1}\omega}\!\!\!\left(\intN|QF(u_\mu)|^\omega\dd y\right)^\frac1\omega.
	\end{equation*}
	Note that the first term is finite since $N\geq2$, while one may bound the second term uniformly with respect to $\mu$ arguing as in the proof of Lemma \ref{Lem:integral-F}: one defines $v_\mu:=H(u_\mu)$ with $H$ as in \eqref{H}, chooses $\varphi=\frac{F(u_\mu)}{f(u_\mu)}$ as test function since $u_\mu\geq0$, and exploits $J_\mu'(u_\mu)[\varphi]=0$ and $J_\mu(u_\mu)=c_\mu$ to infer $\|v_\mu\|^N=Nc_\mu$ and conclude as therein that
	\begin{equation}\label{QFunif_bdd_mu}
		\intN|QF(u_\mu)|^\omega\dd y\leq C
	\end{equation}
	uniformly with respect to $\mu$ provided $\omega\in\left[1,\min\{\nu_*,k_0\}\right)$.
\end{proof}

We have
\begin{equation}\label{Lmu}
	\begin{split}
		L_\mu&(x,y):=\left|\frac{|x-y|^\mu-1}\mu Q(y)F(u_\mu(y))Q(x)f(u_\mu(x))\varphi(x)\chi_{\{|x-y|\leq1\}}(x,y)\right|\\
		&\leq\frac{C_\omega}{|x-y|^{\frac{4(\omega-1)}{3\omega}}}Q(y)F(u_\mu(y))Q(x)f(u_\mu(x))|\varphi(x)|\chi_{\{|x-y|\leq1\}}(x,y):=h(u_\mu)(x,y)
	\end{split}
\end{equation}
and by Lemma \ref{Lem:omega}
\begin{equation*}
	\intN\left(\int_{\{|x-y|\leq 1\}}\frac{Q(y)F(u_\mu(y))}{|x-y|^{\frac{4(\omega-1)}{3\omega}}}\dd y\right)Q(x)f(u_\mu(x))\varphi(x)\dd x\les\intN Qf(u_\mu)\varphi\leq C,
\end{equation*}
where the last uniform bound is deduced arguing as in Lemma \ref{Lem:conv_Riesz_mu}. Hence, applying \cite[Lemma 2.1]{dFMR} one gets $h(u_\mu)\to h(u_0)$ in $L^1(\R^N)$. Since, for a.e. $x,y\in\R^N$
$$L_\mu(x,y)\to L_0(x,y):=\log\frac1{|x-y|}Q(y)F(u_0(y))Q(x)f(u_0(x))\varphi(x)\chi_{\{|x-y|\leq1\}}(x,y),$$
by the dominated converge theorem we infer that
\begin{equation}\label{DCT_1}
	\begin{split}
		&\intN\!\left(\int_{\{|x-y|\leq1\}}\frac{|x-y|^\mu-1}\mu Q(y)F(u_\mu(y))\dd y\right)Q(x)f(u_\mu(x))\varphi(x)\dd x\\
		&\quad\to\intN\!\left(\int_{\{|x-y|\leq1\}}\log\frac1{|x-y|} Q(y)F(u_0(y))\dd y\right)Q(x)f(u_0(x))\varphi(x)\dd x\,.
	\end{split}
\end{equation}
\vskip0.2truecm
Let us now address to the complementary case $|x-y|\geq1$. The uniform decay at $\infty$ of the family $(u_\mu)_\mu$ is provided us by the radial Lemma \ref{RadialLemma}. In fact, using (A) and the density of $C^1_{0,\rad}(\R^N)$ in $\Erad$, by \eqref{unif_bdd_u_mu} this yields
\begin{equation*}%\label{decay}
	|u_\mu(x)|\leq C\|u_\mu\||x|^{-\frac\ell{N-1}}\leq C|x|^{-\frac\ell{N-1}}\quad\ \mbox{for all}\ \,|x|\geq r_0.
\end{equation*}
Applying the mean value theorem to the function $h_z(\mu)=z^{-\mu}$ with $z=|x-y|$, there exists $\tau=\tau(|x-y|)\in(0,1)$ such that
\begin{equation*}%\label{eqn:Th3}
	0\geq G_\mu(x-y)=\frac{|x-y|^{-\mu}-1}\mu=-|x-y|^{-\tau\mu}\log|x-y|\,,
\end{equation*}
therefore, for all $d>0$ there exists $C_d>0$ such that
\begin{equation}\label{eqn:Th4}
	\begin{split}
		\bigg|\frac{|x-y|^{-\mu}-1}\mu&Q(y)F(u_\mu(y))Q(x)f(u_\mu(x))\varphi(x)\chi_{\{|x-y|>1\}}(x,y)\bigg|\\
		&=|x-y|^{-\tau\mu}\log|x-y|Q(y)F(u_\mu(y))Q(x)f(u_\mu(x))|\varphi(x)| \\
		&\leq C_d(|x|+|y|)^dQ(y)F(u_\mu(y))Q(x)f(u_\mu(x))|\varphi(x)|=:\psi(u_\mu)(x,y)\,.
	\end{split}
\end{equation}
Note that, combining (Q), ($f_2$), and the radial Lemma \ref{RadialLemma}, one has
\begin{equation*}
	Q(y)F(u_\mu(y))\les|y|^{b-\frac{\ell\tp}{N-1}}\quad\ \mbox{for}\ \,|y|>R
\end{equation*}
with $R$ sufficiently large. Hence,
\begin{equation}\label{integral_big}
	\begin{split}
		\int_{\R^{2N}}\!\psi(u_\mu)\dd x\dd y&\les C_{R,d}\left(\int_{\{|y|\leq R\}}Q(y)F(u_\mu(y))\dd y\right)\left(\intN Q(x)f(u_\mu(x))|\varphi(x)|\dd x\right)\\
		&\quad+C\left(\int_{\{|y|>R\}}|y|^{d+b-\frac{\ell\tp}{N-1}}\dd y\right)\left(\intN Q(x)f(u_\mu(x))|\varphi(x)|\dd x\right)\\
		&\leq C_1+C_2\int_1^{+\infty}\rho^{d+b-\frac{\ell\tp}{N-1}+N-1}\dd\rho\,.
	\end{split}
\end{equation}
Since $d>0$ is arbitrary and $\tp>\gamma$, the integrability condition reduces to
\begin{equation}\label{integr_cond}
	b-\frac{\ell\gamma}{N-1}+N<0\,,
\end{equation}
where we recall that $\gamma$ is defined in \eqref{gamma}. Hence we need to distinguish two cases:
\begin{itemize}
	\item if $b<\ell-N$, then $\gamma=N$ and the condition \eqref{integr_cond} becomes $b<\frac{\ell N}{N-1}-N$. However, this is automatically satisfied, since it is easy to verify that $\ell-N<\frac{\ell N}{N-1}-N$ for all $\ell>0$ and $N\geq2$;
	\item if $b\geq\ell-N$, then $\gamma=\frac{(b-\ell+N)(N+1)}\ell+N$ and \eqref{integr_cond} becomes
	$$b<(b-\ell +N)\frac{N+1}{N-1}+\frac{\ell N}{N-1}-N\,,$$
	which can be simplified as $2b>\ell-2N$. Again, this is automatically satisfied, since $2(\ell-N)\geq\ell-2N$ for all $\ell>0$.
\end{itemize}
All in all, from \eqref{integral_big} we conclude that $\int_{\R^{2N}}\psi(u_\mu)\leq C$, and analogously one also has $\int_{\R^{2N}}\psi(u_0)<+\infty$. Hence, again \cite[Lemma 2.1]{dFMR} yields
\begin{equation}\label{DCT_2}
	\begin{split}
		&\intN\!\left(\int_{\{|x-y|>1\}}\frac{|x-y|^\mu-1}\mu Q(y)F(u_\mu(y))\dd y\right)Q(x)f(u_\mu(x))\varphi(x)\dd x\\
		&\quad\to\intN\!\left(\int_{\{|x-y|>1\}}\log\frac1{|x-y|} Q(y)F(u_0(y))\dd y\right)Q(x)f(u_0(x))\varphi(x)\dd x\,.
	\end{split}
\end{equation}
Combining \eqref{DCT_1} and \eqref{DCT_2} with \eqref{umu_crit_pt} and \eqref{convergences} we can finally infer that $J'(u_0)=0$.

Next, we show \eqref{logFF}. By Fatou's lemma, we have
\begin{equation}\label{eqn:Th6}
	\begin{split}
		&\bigg|\intN\!\left(\intN\log|x-y|Q(y)F(u_0(y))\dd y\right)Q(x)F(u_0(x))\dd x\bigg|\\
		&\leq\liminf_{\mu\to0}\bigg(\int\int_{\{|x-y|\leq1\}}G_\mu(x-y)Q(y)F(u_\mu(y))\dd y\,Q(x)F(u_\mu(x))\dd x\\
		&-\int\int_{\{|x-y|\geq1\}}G_\mu(x-y)Q(y)F(u_\mu(y))\dd y\,Q(x)F(u_\mu(x))\dd x\bigg).
	\end{split}
\end{equation}
Using \eqref{Lmu}, Lemma \ref{Lem:omega}, and \eqref{QFunif_bdd_mu} with $\omega=1$, we have
\begin{equation*}%\label{eqn:Th7}
	\aligned
	\int&\int_{\{|x-y|\leq1\}}G_\mu(x-y)Q(y)F(u_\mu(y))\dd y\,Q(x)F(u_\mu(x))\dd x\\
	&\leq\intN\left(\int_{\{|x-y|\leq1\}}\frac{Q(y)F(u_\mu(y))}{|x-y|^\frac{4(\omega-1)}{3\omega}}\dd y\right)Q(x)F(u_\mu(x))\dd x\leq C
	\endaligned
\end{equation*}
uniformly for $\mu\in\big(0,\frac{4(\omega-1)}{3\omega}\big)$. So, recalling that $J_\mu(u_\mu)=c_\mu$, we deduce
\begin{equation}\label{eqn:Th8}
	\aligned
	&\frac{C_N}2\intN\left(\int_{\{|x-y|>1\}}G_{\mu}(x-y)Q(y)F(u_\mu(y))\dd y\right)Q(x)F(u_\mu(x))\dd x\\
	&\leq c_\mu+\frac{C_N}2\intN\left(\int_{\{|x-y|\leq1\}}G_{\mu}(x-y)Q(y)F(u_\mu(y))\dd y\right)Q(x)F(u_\mu(x))\dd x-\frac{\|u_\mu\|^N}N\leq C
	\endaligned
\end{equation}
uniformly for $\mu$ sufficiently small by Lemma \ref{MP_level}. Combining \eqref{eqn:Th6}-\eqref{eqn:Th8}, we have
\begin{equation*}%\label{eqn:Th9}
	\bigg|\intN\!\left(\intN\log|x-y|Q(y)F(u_0(y))\dd y\right)Q(x)F(u_0(x))\dd x\bigg|<+\infty\,.
\end{equation*}
This in particular implies that $J(u_0)$ is well-defined, and hence, by Palais' principle of symmetric criticality, we can conclude that $u_0\in\Erad$ is a critical point of $J$ in $E$ that is, $u_0\in E$ solves equation \eqref{Choq_log}.

\vskip0.2truecm
To conclude the proof, we need to show that $u_0\neq0$. Assume by contradiction that $u_0$, hence $u_\mu\rightharpoonup0$ weakly in $E$ and strongly in $L_Q^t(\R^N)$ for $t>\gamma$. We claim that $u_\mu\to0$ strongly in $E$, by showing that $\|u_\mu\|\to0$ as $\mu\to0$. Indeed,
\begin{equation}\label{norm_to0}
	\begin{split}
		0&=J_\mu'(u_\mu)[u_\mu]=\|u_\mu\|^N-C_N\intN\left(G_\mu(\cdot)\ast QF(u_\mu)\right)Qf(u_\mu)u_\mu\dd x\\
		&\geq\|u_\mu\|^N-C_N\int_{\{|x-y|\leq1\}}\left(G_\mu(\cdot)\ast QF(u_\mu)\right)Qf(u_\mu)u_\mu\dd x\\
		&\geq\|u_\mu\|^N-C_N\intN\left(\int_{\{|x-y|\leq1\}}\frac{Q(y)F(u_\mu(y))}{|x-y|^\frac{4(\omega-1)}{3\omega}}\dd y\right)Q(x)F(u_\mu(x))\dd x\\
		&\geq\|u_\mu\|^N-C_NC\intN Q(x)F(u_\mu(x))\dd x.
	\end{split}
\end{equation}
By \eqref{F-C-above} one has
\begin{equation}\label{convQFto0}
	\left|\intN QF(u_\mu)\dd x\right|\les\intN Q|u_\mu|^\tp+\left(\intN Q|u_\mu|^{p\theta'}\right)^\frac1{\theta'}\left(\intN Q\,\Phi_{\alpha\widetilde\theta,j_0}(u_\mu)\right)^\frac1\theta=o_\mu(1),
\end{equation}
with $\widetilde\theta>\theta$, since $u_\mu\to0$ in $L_Q^t(\R^N)$ for $t>\gamma$ and $p,\tp>\gamma$, and by the fact that the last term can be bounded independently of $\mu$ arguing as for \eqref{QFunif_bdd_mu}. Hence, from \eqref{norm_to0} we infer $\|u_\mu\|=o_\mu(1)$. However, from Remark \ref{rem_cMP_unif}, \eqref{eqn:Th4}, and \eqref{QFunif_bdd_mu}, we get
\begin{equation*}\label{}
	\begin{split}
		\underline c&\leq c_\mu=J_\mu(u_\mu)=\frac{\|u_\mu\|^N}N-\frac{C_N}2\intN\left(G_\mu(\cdot)\ast QF(u_\mu)\right)QF(u_\mu)\dd x\\
		&\les-\int_{\{|x-y|\geq1\}}\left(G_\mu(\cdot)\ast QF(u_\mu)\right)QF(u_\mu)\dd x+o_\mu(1)\\
		&\les\intN\intN(|x|+|y|)^dQ(y)F(u_\mu(y))Q(x)F(u_\mu(x))\dd y\dd x+o_\mu(1)\\
		&\les 2\left(\intN|x|^dQ(x)F(u_\mu(x))\dd x\right)\left(\intN Q(y)F(u_\mu(y))\right)+o_\mu(1)\\
		&\les\int_{\{|x|>R\}}|x|^dQ(x)F(u_\mu(x))\dd x+R^d\intN Q(x)F(u_\mu(x))\dd x+o_\mu(1)\\
		&\les\int_{\{|x|>R\}}|x|^dQ(x)F(u_\mu(x))\dd x+o_\mu(1)
	\end{split}
\end{equation*}
by \eqref{convQFto0}. Now we can argue as in \eqref{integral_big}, obtaining
\begin{equation*}%\label{convto0largeR}
	\int_{\{|x|>R\}}|x|^dQ(x)F(u_\mu(x))\dd x\les\int_R^{+\infty}\rho^{d+b-\frac{\ell\tp}{N-1}+N-1}\dd\rho=CR^{d+b-\frac{\ell\gamma}{N-1}+N}\to0
\end{equation*}
as $R\to+\infty$, since the exponent is negative, see \eqref{integr_cond}. This readily yields a contradiction since $\underline c>0$, and shows that $u_0$ is nontrivial.

\section{Existence for the Schr\"odinger-Poisson system \eqref{SP_0}: Proof of Theorem \ref{Thm_SP}}\label{Sec_SP}

Let $u\in E$ be the weak solution of the Choquard equation \eqref{Choq_log} given by Theorem \ref{Thm_log} and define
$$\phi_u(x):=C_N\intN\log\left(\frac1{|x-y|}\right)Q(y)F(u(y))\dd y\,.$$
Note that, buying the lines of Lemma \ref{Lem:conv_Riesz_mu}, one can show that
\begin{equation*}%\label{QF_conv_0}
	\intN QF(u_\mu)\dd x\to\intN QF(u)\dd x\,,
\end{equation*}
as $\mu\to0$, where $(u_\mu)_\mu\subset E$ is the sequence of solutions of the family of approximating problems \eqref{Choq_approx} with $\mu\in(0,\mu_0)$. This in particular implies that
\begin{equation}\label{QF_conv_0}
	\intN QF(u)\dd x<+\infty\,.
\end{equation}
By the radial Lemma \ref{RadialLemma} we can actually improve it.

\begin{lemma}\label{logFu}
	Let $u\in\Erad$ be the weak solution of \eqref{Choq_log} given by Theorem \ref{Thm_log}. Then
	\begin{equation*}%\label{logQF}
		\intN \log(1+|x|)Q(x)F(u(x))\dd x<+\infty\,.
	\end{equation*}
\end{lemma}

\begin{proof}
	By Lemma \ref{RadialLemma} there exists $r_0>0$ such that $|u(x)|\les\|u\||x|^{-\frac\ell{N-1}}$, hence by ($f_2$) and (Q), up to a bigger $r_0$, we have
	\begin{equation*}
		Q(x)F(u(x))\les|x|^{b-\frac{\ell\tp}{N-1}}\quad\ \mbox{for}\ \,|x|>r_0\,.
	\end{equation*}
	Therefore, by \eqref{QF_conv_0} and the elementary estimate $\log(1+|x|)\les|x|^d$ for $|x|>r_0$, $d>0$, we infer
	\begin{equation*}
		\begin{split}
			\intN\log(1+|x|)Q(x)F(u(x))\dd x
			&\les C\log(1+r_0)\intN Q(x)F(u(x))\dd x\\
			&\quad+\int_{|x|\geq r_0}\log(1+|x|)|x|^{b-\frac{\ell\tp}{N-1}}\dd x\\
			&\les 1+\int_{r_0}^{+\infty}\rho^{d+b-\frac{\ell\tp}{N-1}+N-1}\dd\rho<+\infty\,,
		\end{split}
	\end{equation*}
	 for small values of $d$, since the integrability condition \eqref{integr_cond} is fulfilled.
\end{proof}

We are now in the position to prove Theorem \ref{Thm_SP} following the approach of \cite{BCT,CLR}. First, we need to show that $\phi_u\in L_s(\R^N)$, $s>0$:
\begin{equation*}
	\begin{split}
		\frac1{C_N}\intN&\frac{|\phi_u(x)|}{1+|x|^{N+2s}}\dd x\leq\intN Q(y)F(u(y))\left(\intN\left|\log\frac1{|x-y|}\right|\frac1{1+|x|^{N+2s}}\dd x\right){\rm d}y\\
		&\leq\intN Q(y)F(u(y))\left(\int_{\{|x-y|>1\}}\frac{\log|x-y|}{1+|x|^{N+2s}}\dd x+\int_{\{|x-y|\leq1\}}\log\frac1{|x-y|}\dd x\right){\rm d}y\\
		&\leq\left(\intN Q(y)F(u(y))\,\dd y\right)\left(\intN\frac{\log(1+|x|)}{1+|x|^{N+2s}}\,\dd x\right)\\
		&\quad+\left(\intN\log(1+|y|)Q(y)F(u(y))\,\dd y\right)\left(\intN\frac{\dd x}{1+|x|^{N+2s}}\right)\\
		&\quad+\|\log(\cdot)\|_{L^1(B_1(0))}\intN Q(y)F(u(y))\dd y<+\infty
	\end{split}
\end{equation*}
for all $s>0$, using the elementary estimate $\log|x-y|\leq\log(1+|x|)+\log(1+|y|)$ for $|x-y|>1$, and Lemma \ref{logFu}.

Define now the function
$$\tw_u(x):=C_N\intN\log\left(\frac{1+|y|}{|x-y|}\right)Q(y)F(u(y))\dd y,$$
which we know by \cite[Lemma 2.3]{H} being a solution in the sense of Definition \ref{sol_Poisson} of $(-\Delta)^{\frac N2}\tw_u=\ff(x)$ in $\R^N$, where $\ff:=QF(u)\in L^1(\R^N)$ by \eqref{QF_conv_0}, and compute the difference 
\begin{equation*}
	\begin{split}
		\tw_u(x)-\phi_u(x)&=C_N\intN\left(\log\left(\frac{1+|y|}{|x-y|}\right)-\log\left(\frac1{|x-y|}\right)\right)Q(y)F(u(y))\dd y\\
		&=C_N\intN\log(1+|y|)Q(y)F(u(y))\dd y<+\infty\,,
	\end{split}
\end{equation*}
that is constant, by Lemma \ref{logFu}. This implies that $\phi_u$ is a solution of \eqref{SP_0} in the sense of Definition \ref{sol_SP}, by applying \cite[Lemma 2.4]{H}, for which all such solutions of $(-\Delta)^{\frac N2}\phi=\ff$ in $\R^N$ are of the form $\phi=\tw_u+p$ with $p$ polynomial of degree at most $N-1$.

\vskip0.4truecm
\paragraph{Acknowledgements:} The Author is member of \textit{Gruppo Nazionale per l'Analisi Matematica, la Probabilità e le loro Applicazioni} (GNAMPA) of the \textit{Instituto Nazionale di Alta Matematica} (INdAM), and was partly supported by INdAM-GNAMPA Project 2023 titled \textit{Interplay between parabolic and elliptic PDEs} (codice CUP E53C2200l93000l).


\begin{thebibliography}{ChSTW}
	
	\bibitem[AFS]{AFS} Albuquerque F.S.B., Ferreira M.C., Severo U.B.
	{\em Ground state solutions for a nonlocal equation in $\R^2$ involving vanishing potentials and exponential critical growth.} Milan J. Math. 89 (2021), 263-294.
	
	\bibitem[ACTY]{ACTY} Alves C.O., Cassani D., Tarsi C., Yang M.
	{\em Existence and concentration of ground state solutions for a critical nonlocal Schrödinger equation in $\R^2$.} J. Differential Equations 261 (2016), no. 3, 1933-1972.
	
	\bibitem[AF]{AF} Alves C.O., Figueiredo G.
	{\em Existence of positive solution for a planar Schr\"odinger-Poisson system with exponential growth.} J. Math. Phys. 60 (2019), no. 1, 011503, 13 pp.
	
	\bibitem[ASM]{ASM} Alves C.O., Souto M.A.S., Montenegro M.
	{\em Existence of solution for two classes of elliptic problems in $\R^N$ with zero mass.} J. Differential Equations 252 (2012), no.10, 5735-5750.
	
	\bibitem[AY]{AY} Alves C.O., Yang J.
	{\em Existence and regularity of solutions for a Choquard equation with zero mass.} Milan J. Math. 86 (2018), no.2, 329-342.
	
	\bibitem[AP]{AP} Azzollini A., Pomponio A.
	{\em On a ``zero mass'' nonlinear Schrödinger equation.} Adv. Nonlinear Stud. 7 (2007), no.4, 599-627.
	
	\bibitem[BF]{BF} Benci V., Fortunato D.
	{\em Variational Methods in Nonlinear Field Equations}, Springer 2014.
	
	\bibitem[BL]{BL} Berestycki H., Lions P.-L.
	{\em Nonlinear scalar field equations. I. Existence of a ground state.}	Arch. Rational Mech. Anal. 82 (1983), no. 4, 313-345.
	
	\bibitem[BCV]{BCV} Bonheure D., Cingolani S., Van Schaftingen J.
	{\em The logarithmic Choquard equation: sharp asymptotics and nondegeneracy of the groundstate.} J. Funct. Anal. 272 (2017), no. 12, 5255-5281.
	
	\bibitem[BCT]{BCT} Bucur C.D., Cassani D., Tarsi C.
	{\em Quasilinear logarithmic Choquard equations with exponential growth in $\R^N$.} J. Differential Equations 328 (2022), 261-294.
	
	\bibitem[BM1]{BM1} de S. B\"oer M., Miyagaki O.H.
	{\em The Choquard logarithmic equation involving a nonlinearity with exponential growth.} Topol. Methods Nonlinear Anal. 60 (2022), no.1, 363-385.
	
	\bibitem[BM2]{BM2} de S. B\"oer M., Miyagaki O.H.
	{\em (p,N)-Choquard logarithmic equation involving a nonlinearity with exponential critical growth: existence and multiplicity.} Preprint 2021, \url{https://arxiv.org/pdf/2105.11442.pdf}.
	
	\bibitem[CMR]{CMR} Carvalho J.L., Medeiros E., Ribeiro, B.
	{\em On a planar Choquard equation involving exponential critical growth.} Z. Angew. Math. Phys. 72 (2021), no. 6, Paper No. 188, 19 pp.
	
	\bibitem[CDL]{CDL} Cassani D., Du L., Liu Z.
	{\em Positive solutions to the planar Choquard equation via asymptotic approximation}, Preprint 2023, \url{https://arxiv.org/pdf/2305.10905.pdf}.
	
	\bibitem[CLR]{CLR} Cassani D., Liu L., Romani G.
	{\em Nonlocal planar Schr\"odinger-Poisson systems in the fractional Sobolev limiting case.} Preprint 2023, available at arxiv:2305.15274.
	
	\bibitem[CT]{CT} Cassani D., Tarsi C.
	{\em Schr\"odinger-Newton equations in dimension two via a Pohozaev-Trudinger log-weighted inequality.} Calc. Var. Partial Differential Equations 60 (2021), no. 5, Paper No. 197, 31 pp.
	
	\bibitem[CVZ]{CVZ} Cassani D., Van Schaftingen J., Zhang J.
	{\em Groundstates for Choquard type equations with Hardy-Littlewood-Sobolev lower critical exponent.} Proc. R. Soc. Edinb. Sect. A 150 (2020), 1377-1400.
	
	\bibitem[CZ]{CZ} Cassani D., Zhang J.
	{\em Choquard-type equations with Hardy-Littlewood-Sobolev upper-critical growth.} Adv. Nonlinear Anal. 8 (2019), 1184-1212.
	
	\bibitem[ChSTW]{CSTW} Chen S., Shu M., Tang X., Wen L.
	{\em Planar Schr\"odinger-Poisson system with critical exponential growth in the zero mass case.} J. Differential Equations 327 (2022), 448-480.
	
	\bibitem[ChT]{ChT} Chen S., Tang X.
	{\em On the planar Schr\"odinger-Poisson system with zero mass and critical exponential growth.} Adv. Differential Equations 25 (2020), no.11-12, 687-708.
	
	\bibitem[CW]{CW} Cingolani S., Weth T.
	{\em On the planar Schr\"odinger-Poisson system.} Ann. Inst. H. Poincar\'{e} C Anal. Non Lin\'{e}aire 33 (2016), no. 1, 169-197.
	
	\bibitem[CW2]{CW2} Cingolani S., Weth T.
	{\em Trudinger-Moser-type inequality with logarithmic convolution potentials.} J. Lond. Math. Soc. (2) 105 (2022), no. 3, 1897-1935.

	\bibitem[dAC]{dAC} de Albuquerque J.C., Carvalho J.L.
	{\em Quasilinear equation with critical exponential growth in the zero mass case.} Nonlinear Analysis 232 (2023) 113286.
	
	\bibitem[dACS]{dACS} de Albuquerque, J.C., Carvalho, J.L., Souza Filho, A.P.F.
	{\em On a quasilinear logarithmic  N -dimensional equation involving exponential growth.} J. Math. Anal. Appl. 519 (2023), no. 1, Paper No. 126751, 23 pp.

	\bibitem[dFMR]{dFMR} de Figueiredo D., Miyagaki O., Ruf B.
	{\em Elliptic equations in $\R^2$ with nonlinearities in the critical growth range}, Calc. Var. Partial Differential Equations 3 (1995), 139-153.
	
	\bibitem[DW]{DW} Du M., Weth T.
	{\em Ground states and high energy solutions of the planar Schr\"odinger-Poisson system.} Nonlinearity 30 (2017), no. 9, 3492-3515.
	
	\bibitem[DS]{DS} Dunford N., Schwartz J.T.
	{\em Linear Operators. I. General Theory.} Pure Appl. Math., Vol. 7, Interscience Publishers, Inc., New York, 1958.
	
	\bibitem[E]{Ekeland} Ekeland I.
	{\em Convexity methods in Hamiltonian Mechanics.} Springer (1990).
	
	\bibitem[Ga]{G} Galdi G.P.
	{\em An introduction to the mathematical theory of the Navier-Stokes equations. Steady-state problems.} Second edition. Springer Monographs in Mathematics. Springer, New York, 2011.
	
	\bibitem[Gi]{Gi} Gidas B.
	{\em Euclidean Yang-Mills and related equations.} in: Bifurcation Phenomena in Mathematical Physics and Related Topics, Proc. NATO Advanced Study Inst., Carg\`{e}se, 1979, in: NATO Adv. Study Inst. Ser. C: Math. Phys. Sci., vol. 54, Reidel, Dordrecht-Boston, Mass, 1980
	
	\bibitem[H]{H} Hyder A.
	{\em Structure of conformal metrics on $\R^n$ with constant $Q$-curvature.} Differential Integral Equations 32 (2019), no. 7-8, 423-454.
	
	\bibitem[LY]{LY} Li Q., Yang Z.
	{\em Multiple solutions for a class of fractional quasi-linear equations with critical exponential growth in $\R^N$.} Complex Var. Elliptic Equ. 61 (2016), no.7, 969-983.
	
	\bibitem[LL]{LL} Lieb E.H., Loss M.
	{\em Analysis.} Second edition, Graduate Studies in Mathematics 14. American Mathematical Society, Providence, RI (2001).
	
	\bibitem[LRZ]{LRZ} Liu Z., R\u{a}dulescu V.D., Zhang J.
	{\em A planar Schr\"odinger-Newton system with Trudinger-Moser critical growth.} Calc.Var. P.D.E. 62 (2023), 31pp.
	
	\bibitem[LRTZ]{LRTZ} Liu Z., R\u{a}dulescu V.D., Tang C., Zhang J.
	{\em Another look at planar Schr\"odinger-Newton systems.} J. Differential Equations 328 (2022), 65-104. 
	
	\bibitem[MV1]{MV1} Moroz V., Van Schaftingen J.
	{\em Groundstates of nonlinear Choquard equations: existence, qualitative properties and decay asymptotics.} J. Funct. Analysis 265 (2013) 153-184.
	
	\bibitem[MV2]{MV} Moroz V., Van Schaftingen J.
	{\em A guide to the Choquard equation.} J. Fixed Point Theory Appl. 19 (2017), 773-813. 

	\bibitem[PS]{PS} Pucci P., Serrin J.
	{\em The strong maximum principle revisited.} J. Differential Equations 196 (2004), no. 1, 1-66.
	
	\bibitem[Si]{Simon} Simon J.
	{\em R\'{e}gularit\'{e} de la solution d'une \'{e}quation non lin\'{e}aire dans $\R^N$.} Lectures Notes in Math. No. 665, Berlin: Springer (1978).
	
	\bibitem[St]{Stubbe} Stubbe J.
	{\em Bound states of two-dimensional Schr\"odinger-Newton equations.}
	Preprint 2008, \url{https://arxiv.org/pdf/0807.4059v1.pdf}.
    
    \bibitem[WCR]{WCR} Wen L., Chen S., R\u{a}dulescu, V.D.
    {\em Axially symmetric solutions of the Schr\"odinger-Poisson system with zero mass potential in $\R^2$.} Appl. Math. Lett. 104 (2020), 106244, 7 pp.
	
\end{thebibliography}
\end{document}